\documentclass[11pt,a4paper]{amsart}

\usepackage{soul} 
\usepackage{enumerate}

\makeatletter
\newcommand{\xRightarrow}[2][]{\ext@arrow 0359\Rightarrowfill@{#1}{#2}}
\makeatother

\makeatletter
\@addtoreset{equation}{section}
\makeatother
\numberwithin{equation}{subsection}

\usepackage[dvipsnames]{xcolor}

\newcommand{\dd}{/\hspace*{-0.5mm}/}
\newcommand{\op}{\text{op}}
\def\Card{\mathrm{Card}}

\def\Set{{\mathsf{Set}}}
\def\TTop{{\mathsf{Top}}}
\def\Pos{{\mathsf{Pos}}}
\def\Preord{{\mathsf{Preord}}}
\def\Cat{{\mathsf{Cat}}}
\def\CAT{{\mathsf{CAT}}}

\def\Grp{\mathsf{Grp}}

\def\tCat{{2\textrm{-}\mathsf{Cat}}}

\def\colim{{\mathrm{colim}}}
\def\functorA{{ \overline{A} }}
\def\functorB{{ \overline{B} }}
\def\functorS{{ \overline{s} }}

\def\Ran{\textrm{Ran}}
\def\Lan{\textrm{Lan}}
\def\YonedaE{\mathcal{Y}}
\def\imagefactorizationofprecomposition{\mathrm{Im}\left( \cv\textrm{-}\Cat\left[  J , \cc\right]  \right) }
\def\imagefactorizationofprecompositionbo{\cv\textrm{-}\Cat\left[  J , \cc\right] _{\mathrm{Im}} }

\usepackage[utf8]{inputenc}
\usepackage[colorinlistoftodos]{todonotes} 
\input xy
\xyoption{all} \xyoption{arc}
\usepackage{graphicx}
\usepackage[all,2cell,cmtip]{xy}
\UseAllTwocells
\usepackage{tikz-cd}
\usepackage{tikz}
\usepackage{setspace}
\usepackage{kpfonts}
\usepackage[draft,marginclue,nomargin,footnote]{fixme}
\usepackage{amsbsy,mathrsfs,latexsym,amssymb,mathbbol,stmaryrd}
\theoremstyle{plain}
\newtheorem{theorem}{Theorem}[section]
\newtheorem{proposition}[theorem]{Proposition}
\newtheorem{corollary}[theorem]{Corollary}
\newtheorem{lemma}[theorem]{Lemma}
\theoremstyle{definition}
\newtheorem{definition}[theorem]{Definition}

\newtheorem{example}[theorem]{Example}
\newtheorem{examples}[theorem]{Examples}
\newtheorem{assumption}[theorem]{Assumption}

\newtheorem{remark}[theorem]{Remark}

\newtheorem{notation}[theorem]{Notation}

\newtheorem{question}[theorem]{Question}

\usepackage{cite}

\def\co{{\mathrm{co}}}

\def\id{\mathrm{id}}

\def\N{{\rm I\kern-.20em N}}
\def\R{{\rm I\kern-.17em R}}
\def\Z{{\rm Z\kern-.32em Z}}

\newcommand{\gl}{\lambda}
\newcommand{\gb}{\beta}
\newcommand{\ga}{\alpha}

\newcommand{\wrt}{ with respect to }

\newcommand\ca{\mathcal{A} }
\newcommand\cb{\mathcal{B}}
\newcommand\cc{\mathcal{C} }
\newcommand\cd{\mathcal{D} }

\newcommand\cv{\mathcal{V} }
\newcommand\cvhf{\mathcal{V}\textrm{-} }

\newcommand\DisFib{\mathtt{DisFib}}
\newcommand\Fib{\mathtt{Fib}}

\newcommand\cm{\mathcal {M}}
\newcommand\ce{\mathcal {E}}
\newcommand\ci{\mathcal {I}}

\newcommand\Ef{\ce\!\!\mid \!\! f}

\newcommand\A{\mathbb{A}}
\newcommand\B{\mathbb{B}}
\newcommand\C{\mathbb{C}}
\renewcommand\S{\mathbb{S}}
\renewcommand\D{\mathbb{D}}
\newcommand\E{\mathbb{E}}

\begin{document}
%
\FXRegisterAuthor{fl}{aja}{\color{magenta}FL}
\FXRegisterAuthor{ls}{als}{\color{blue}LS}

\title[On lax epimorphisms and the associated factorization]
{On lax epimorphisms and the associated factorization}

\author{Fernando Lucatelli Nunes}
\address{Utrecht University, Netherlands}
\email{f.lucatellinunes@uu.nl}
\author{Lurdes Sousa}
\address{ University of Coimbra, CMUC, Department of Mathematics, Portugal  \& Polytechnic Institute of Viseu, ESTGV, Portugal}
\email{sousa@estv.ipv.pt}
\thanks{This research was partially supported  by the Centre for Mathematics of the University of Coimbra - UIDB/00324/2020, funded by the Portuguese Government through FCT/MCTES}

\date{March 24, 2022}

\begin{abstract}
 We study lax epimorphisms in 2-categories, with special attention to Cat and $\cv$-Cat. We show that any 2-category with convenient colimits has an orthogonal $LaxEpi$-factorization system, and we give a concrete description of this factorization in $\Cat$.
\end{abstract}

\maketitle

\section{Introduction}

A morphism $e:A\to B$ in a category $\mathbb{A}$ is an epimorphism if, for every object $C$,  the map $\A(e,C):\A(B,C)\to \A(A,C)$ is injective; looking at the hom-sets as discrete categories, this means that the functor $\A(e,C)$ is fully faithful. \textit{Lax epimorphisms} (also called \textit{co-fully-faithful} morphisms) are a 2-dimensional version of epimorphisms; in a 2-category they are precisely the 1-cells  $e$ making $\A(e,C)$ fully faithful for all $C$.

One of the most known (orthogonal) factorization systems in the category  of small categories and functors is the comprehensive factorization system of Street and Walters \cite{SW}. Another known factorization system consists of bijective-on-objects functors on the left-hand side and fully faithful functors on the right. Indeed in both cases we have an orthogonal factorization system in the 2-category $\Cat$ in the sense of Definition \ref{d:ofs}. This means that with the usual notion in ordinary categories we have a 2-dimensional aspect of the diagonal fill-in property. Here we show that $\Cat$ has also an orthogonal $(\ce, \cm)$-factorization system  where $\ce$ is the class of all lax epimorphisms, and present a concrete description of it, making use of a characterization of the  lax epimorphic functors given in \cite{AeBSV} (Theorem \ref{t:f-Cat}).

Moreover, any 2-category  has an orthogonal $(LaxEpi, LaxStrongMono)$-factorization system provided that it has 2-colimits and is almost cowellpowered with respect to lax epimorphisms (Theorem \ref{t:ofs}). Here to be \textit{almost cowellpowered} with respect to a class $\ce$ of morphisms means that, for every morphism $f$, the category of all factorizations $d\cdot e$ of $f$ with $e\in \ce$ has a weakly terminal set. A key property is the fact that lax epimorphisms are closed under 2-colimits (Theorem \ref{t:stcolim}).

We dedicate the last section to the study of lax epimorphisms in the 2-category $\cv$-$\Cat$ for $\cv$ a complete symmetric monoidal closed category. In this context, it is natural to consider a variation of the notion of lax epimorphism: We say that  a $\cv$-functor $J:\ca\to \cb$ is a $\cv$-lax epimorphism if the $\cv$-functor  $\cv$-$\Cat[J,\cc]: \cv$-$\Cat[\cb,\cc]\to \cv$-$\Cat[\ca,\cc]$  is $\cv$-fully faithful for all small $\cv$-categories $\cc$.
Assuming that $\cv$ is also cocomplete, Theorem \ref{t:V-lax} gives several characterizations of the  lax epimorphisms in the 2-category $\cv$-$\Cat$. In particular, we show that they are precisely the $\cv$-lax epimorphisms, and also precisely those $\cv$-functors for which there is an  isomorphism $\Lan _J\cb (B, J - )\cong \cb (B, -) $ ($\cv$-natural in $B\in\cb ^\op $). Moreover, $\cv$-lax epimorphisms are equivalently defined if above we replace all small $\cv$-categories $\cc$ by all possibly large $\cv$-categories $\cc$, or by just the category $\cv$. This last characterization, as well as Theorem \ref{laxepi-if-and-only-if-absolutelydense}, which characterizes $\cv$-lax epimorphisms as absolutely $\cv$-(co)dense $\cv$-functors, have been proved for $\cv=\Set$ in \cite{AeBSV}.

For the basic theory on 2-categories we refer to \cite{KellyStreet} and \cite{lack10}. For a detailed account of 2-dimensional (co)limits, see \cite{kelly89}; here we use the notation $\lim(W,F)$ for the limit of $F:\A\to \B$ weighted (``indexed'' in Kelly's language) by $W:\A\to \Cat$. Concerning enriched categories, we refer to \cite{kelly82}.

\section{Lax epimorphisms in 2-categories}

In this section we present some basic properties and examples on lax epimorphisms. We end up by showing that, under reasonable conditions, for 2-categories $\S$ and $\B$, every lax epimorphism of the 2-category  2-$\Cat[\S,\B]$ is pointwise. Pointwise lax epimorphisms will have a role in the main result of Section \ref{(LE,LSM)}.

\begin{definition}\label{d:laxepi}
A {\em lax epimorphism} in a 2-category $\A$ is a 1-cell $f:A\to B$ for which all the hom-functors
$$\A(f,C):\A(B,C)\to \A(A,C)$$
(with $C\in \A$) are fully faithful.\end{definition}

\begin{remark}[Duality and Coduality]\label{r:ff}
The notion of lax epimorphism is dual to the the notion of \textit{fully faithful morphism} (in a $2$-category). That is the reason why lax epimorphisms are also called  {\em   co-fully-faithful morphisms}.
Indeed, the notion fully faithful morphism in the 2-category of small categories $\Cat$  coincides with the notion  of  fully faithful functor, since a functor $P:\A\to \B$ is fully faithful if and only if    the functor $\Cat(\C, P):\Cat(\C, \B)\to \Cat(\C, \A)$ is fully faithful for all categories $\C$.

On the other hand, the notion of lax epimorphism is \textit{self-codual}. Namely, a morphism $p : A\to B $ is a lax epimorphism in $\A $ if and only if the corresponding morphism in $\A ^\co $ (the $2$-category obtained after inverting the directions of the $2$-cells in $\A$) is a lax epimorphism.
\end{remark}

\begin{remark}\label{r:1-out-2}\label{rem:closed-isomorphism-classes}
Lax epimorphisms are closed for isomorphism classes. That is to say, if $f\cong g$ and $g $ is a lax epimorphism, then so is $f$. Moreover, we have that lax epimorphisms are closed under  composition and are right-cancellable: for composable morphisms $r$ and $s$,  if $r$ and  $sr$ are lax epimorphisms, so is $s$.
\end{remark}

\begin{examples}\label{es:lax-epi}
\begin{enumerate}
\item\label{lax-epi:epi} In a locally discrete  2-category,  lax epimorphisms are just epimorphisms,    since fully faithful  functors between discrete categories  are   injective functions on the objects. But, in general,  the class of  lax epimorphisms and the one of  epimorphisms are different and no one contains the other (see \cite{AeBSV}).

\item\label{lax-epi:coeq} Coequifiers  are lax epimorphisms. The property of being a lax epimorphism is precisely the two-dimensional aspect of  the universal property of a coequifier (see \cite[pag.~309]{kelly89}). But, as observed in \cite{AeBSV}, coequalizers in $\Cat $ are not necessarily lax epimorphisms.

    \item\label{lax-epi:equi}
     Any equivalence is a lax epimorphism. Recall that a morphism $g:A\to B$ is an equivalence if there is $f:B\to A$ with $gf\cong 1_B$ and $fg\cong 1_A$. This is equivalent to the existence of an adjunction between $f$ and $g$ with both unit and counit being invertible, and it is also well known that it is equivalent to the existence of an adjunction $(\varepsilon, \eta):f\dashv g$ together with both $f$ and $g$ fully faithful. Dually, $g:A\to B$  is an equivalence  if and only if   there is an adjunction $(\varepsilon, \eta):f\dashv g$  with both $f$ and $g$ being lax epimorphisms.
     Moreover, given an adjunction $(\varepsilon, \eta):f\dashv g : A\to B$ in a $2$-category $\A $, the morphism $g$ is a lax epimorphism if and only if $f$ is fully faithful, if and only if $\eta$ is invertible (see \cite[Lemma~2.1]{lack10}).

\item\label{lax-epi:coin} In a locally thin  2-category  (i.e., with the hom-categories  being preordered sets), the lax epimorphisms are the order-epimorphisms, i.e.,  morphisms $f$ for which  $g\cdot f\leq h\cdot f$ implies $g\leq h$; and coinserters are lax epimorphisms -- this immediately follows from the definition of coinserter (see, for instance, \cite[pag.~307]{kelly89}).

    However,  coinserters are not lax epimorphisms in general; we indicate a simple counter-example in the 2-category  $\Cat$ of small categories.\footnote{This rectifies  \cite[Example~2.1.1]{AeBSV}.}   Let $\ca$ be the discrete category  with a unique 	object    $A$,  $\cb$  the discrete category  with two objects, $FA$ and $GA$, and $F,G:\ca\to\cb$ the functors defined  according to the name of the objects     of $\cb$. The coinserter of $F$ and $G$ is an inclusion $P:\cb\to \cc$, where $\cc$ has the same objects     as $\cb$ and a unique non trivial morphism, $\ga_A:FA\to GA$. More precisely, the coinserter is given by the pair $(P, \ga)$. (For a description of  coinserters in $\Cat$, see \cite{BKPS}, Example 6.5.)
But $P$ is not a lax epimorphism. Indeed, let    $J,K:\cc\to \cd$ be two functors, where  the category  $\cd$ consists of four objects      and six non trivial morphisms as in the diagram below, with $K\alpha_A\cdot \gamma_{FA}=r\not=s=\gamma_{GA}\cdot J\alpha_A$:
$$\xymatrix{JFA\ar[rr]^{\gamma_{FA}}\ar@<0.9ex>[rrd]^r\ar@<-1.0ex>[rrd]^*-<0.9em>{^{\not=}}_s\ar[d]_{J\ga_{A}}&&KFA\ar[d]^{K\ga_{A}}\\
JGA\ar[rr]_{\gamma_{GA}}&&KGA}$$
Then, we have a natural transformation $\gamma: JP\to KP$ which cannot be expressed as $\gamma=\overline{\gamma}\ast \id_P$ for any $\overline{\gamma}:J\Rightarrow K$.

\item\label{lax-epi:pos}
In the 2-category $\Pos$ of posets, monotone functions and pointwise order between them, lax epimorphisms coincide with epimorphisms, and also with coinserters of some pair of morphisms (see  \cite[Lemma 3.6]{ASV}).

\item\label{lax-epi:preord} In $\Preord$, lax epimorphisms need not to be  epimorphisms: they are just the monotone maps $f:A\to B$  such that  every $b\in B$ is isomorphic to $f(a)$ for some $a$.

Moreover, coinserters are strictly contained in lax epimorphisms, they are precisely the monotone bijections.  Indeed, given $f,g:A\to B$, let $\bar{B}$ be the  underlying set of $B$ with the preorder given by the reflexive and transitive closure of $\leq_B\cup \leq^{\prime}$, where $\leq_B$ is the order in $B$ and $y\leq^{\prime} z$ whenever there is some $x\in A$ with $y\leq f(x)$ and $g(x)\leq z$; the coinserter is the identity map from $B$ to $\bar{B}$. Conversely, if $h:B\to C$ is a monotone bijection, it is the coinserter of the projections $\pi_1, \pi_2: P\to B$, where $P$ is the comma object of  $h$ along itself.

Observe that the functor $P:\cb\to \cc$ of Example \eqref{lax-epi:coin} is indeed a morphism of the full 2-subcategory $\Preord$ of $\Cat$; it is a lax epimorphism in $\Preord$ but not in $\Cat$.

\item\label{lax-epi:grp} Let $\Grp$ be the 2-category of groups, homomorphisms, and  2-cells from $f$ to $g$ in $\Grp(A,B)$ given by those  elements $\alpha$ of $B$ with  $f(x)\circ \alpha= \alpha \circ g(x)$, for all $x\in A$ (where $\circ$ denotes the group multiplication). The horizontal composition of $\alpha:f\to g$ with $\beta:h\to k:B\to C$ is given by $\beta\ast \alpha=h(\alpha)\circ \beta =\beta\circ k(\alpha)$;  and the unit on an arrow $f:A\to B$ is simply the neutral element of $B$ (see \cite{Borceux1})\footnote{This 2-category is the full subcategory of 2-$\Cat$  of all groupoids with just one object.}.

The lax epimorphisms of $\Grp$ are precisely the regular epimorphisms, that is, surjective homomorphisms. Indeed, given a surjective homomorphism  $f:A\to B$,   homomorphisms $g,h:B\to C$ and an element $\gamma\in C$, the equalities  $g(f(x))\circ \gamma =\gamma \circ (h(f(x))$ for all $x\in A$ imply $g(y)\circ \gamma=\gamma \circ h(y)$ for all $y\in B$, showing that $f$ is a lax epimorphism.  Conversely, given a lax epimorphism   $f:A\to B$, consider its $(RegEpi, Mono)$-factorization   in $\Grp$:
$$\xymatrix{A\ar[r]^{q}&M\ar[r]^m&B}.$$
Since $q$ and $qm$ are lax epimorphisms, so is $m$, by Remark \ref{r:1-out-2}. We show that then $m$ is an isomorphism.
In $\Grp$, monomorphisms are regular (see \cite{AHS}); let $g,h:B\to C$ be a pair whose equalizer is the inclusion $m:M\hookrightarrow B$, that is, $M=\{y\in B\mid g(y)=h(y)\}$. Denoting the neutral element of $C$ by $e$, we have a 2-cell $e: gm\to hm$. Since $m$ is a lax epimorphism,   there is a unique $\alpha:g\to h$ with $\alpha\ast e=e$. But $\alpha\ast e=g(e)\circ \alpha=\alpha\circ h(e)=\alpha$; hence $\alpha=e$, that is, $g(y)\circ e=e\circ h(y)$ for all $y\in B$. Thus, $B=M$ and $m$ is the identity morphism.
\end{enumerate}
\end{examples}

\begin{remark}
Recall that a $2$-functor $G : \A\to\B$ is \textit{locally fully faithful} if, for any $A,B\in\A $, the
functor $G _{A,B} : \A\left( A, B \right)\to \B\left( G (A), G(B)\right)    $ is fully faithful.

It is natural to consider lax epimorphisms in the context of  $2$-adjunctions or biadjunctions. Let
$ \left( \varepsilon , \eta \right): F\dashv G:\A\to \B $
be a $2$-adjunction (respectively, biadjunction).
In this case, we have that, for any $A,B \in \A $,
\begin{equation}\label{eq:adjunction-diagram-counit}
\begin{tikzpicture}[x=7cm, y=1.5cm]
\node (a) at (0,0) {$\A\left( A, B\right)  $};
\node (b) at (1, 0) { $\B\left( G(A), G(B) \right)  $};
\node (c) at (1,-1) {$\A \left( FG(A), B\right) $};
\draw[->] (a)--(c) node[midway,below left] {$\A \left( \varepsilon _ A, B\right) $};;
\draw[->] (a)--(b) node[midway,above] {$G_{A,B}  $};
\draw[->] (b)--(c) node[midway,right] {$\chi _{ G(A) ,B } $};
\end{tikzpicture}
\end{equation}
commutes (respectively, commutes up to an invertible natural transformation), in which \begin{eqnarray*}
\chi _{ G(A) ,B } : &\B \left( G (A), G (B) \right) & \rightarrow\A \left( FG (A), B \right) \\
                                 & h & \mapsto   \varepsilon  _{B}\circ F (h)
\end{eqnarray*}
is the invertible functor (respectively, equivalence) of the $2$-adjunction (biadjunction).

In the situation above, since isomorphisms (respectively, equivalences) are fully faithul and fully faithful functors are left-cancellable
(see Remark \ref{r:1-out-2}),  we have that
$G_{A,B} : \A \left( A,B\right) \to \B \left( G(A), G(B)\right) $ is fully faithful if, and only if, $\A \left( \varepsilon _ A, B\right) $ is fully faithful. Therefore, the $2$-functor $G: \A\to\B $ is locally fully faithful if and only if $\varepsilon _ C $ is a lax epimorphism for every $C\in \A $.
\end{remark}

\begin{remark}\label{rem:laxepimorphisms-characterization-opcomma-objects}
It is known that in a 2-category with cotensor products, fully faithful  morphisms are those $p:A\to B$ such that the comma object of $p$ along itself is isomorphic to the cotensor product $2\pitchfork A$. Dually, assuming the existence of tensor products, a morphism $p:A\to B$ is a
 lax epimorphism if and only if
\begin{equation*}
\begin{tikzpicture}[x=3cm, y=2cm]
\node (a) at (0,0) {$A $};
\node (b) at (1, 0) { $B$ };
\node (c) at (1,-1) {$\mathsf{2}\otimes B $};
\node (d) at (0,-1) {$ B $};
\draw[->] (a)--(b) node[midway,above] {$ p  $};
\draw[->] (b)--(c) node[midway,right] {$ \nu _1  $};
\draw[->] (a)--(d) node[midway,left] {$ p  $};
\draw[->] (d)--(c) node[midway,below] {$ \nu _0  $};
\draw[double,->] (0.3,-0.6)--(0.7,-0.6) node[midway,above] {$\alpha\ast \id _ p $};
\end{tikzpicture}
\end{equation*}
is an opcomma object, in which
\begin{equation*}
\begin{tikzpicture}[x=2.5cm, y=0.5cm]
\node (a) at (0,0) {$B $};
\node (c) at (2,0) {$\mathsf{2}\otimes B $ };
\node (g) at (1,-2.1) {$\nu _ 0 $};
\node (f) at (1,2.1) {$\nu _1 $};
\draw[->] (a) to[bend right] (c);
\draw[->] (a) to[bend left] (c);
\draw[double,->] (1,-1)--(1,1) node[midway,right] {$\alpha $};
\end{tikzpicture}
\end{equation*}
is the tensor product.
\end{remark}

Since, in the presence of tensor products, lax epimorphisms are characterized by opcomma objects as above, we conclude that:
\begin{lemma}\label{lem:creation-of-colimits-vs-laxepi}
Let $F : \B\to\A $ be a $2$-functor.
\begin{enumerate}[1.]
    \item Assuming that $\B $ has tensor products, if $F$ preserves opcomma objects and tensor products, then $F$ preserves lax epimorphisms.
    \item Assuming that $\A $ has tensor products, if $F$ creates opcomma objects and tensor products, then $F$ reflects lax epimorphisms.
\end{enumerate}
\end{lemma}

Moreover, we also have that:
\begin{lemma}\label{lem:2-adjunction-preservation-of-lax-epi}
Let $F \dashv G $ be a $2$-adjunction.
\begin{enumerate}[(1)]
    \item \label{Lemma:lax-epi-2adjunction(a)} The $2$-functor $F : \B\to\A $ preserves lax epimorphisms.
    \item \label{Lemma:lax-epi-2adjunction(b)} If $G$ is essentially surjective, then $F$ reflects
    lax epimorphisms.
\end{enumerate}
\end{lemma}
\begin{proof}
For any object $W$ of $\A $ and any morphism $p: A\to B$ of $\B $, the diagram
\begin{equation}\label{eq:diagram-chi-2adjunction}
\begin{tikzpicture}[x=7cm, y=2cm]
\node (c) at (0,0) {$\A \left( F(B), W \right) $};
\node (d) at (1, 0) { $\A \left( F(A) , W\right)   $ };
\node (a) at (1,-1) {$\B \left( A, G(W)  \right)  $};
\node (b) at (0,-1) {$ \B \left( B, G(W)  \right)  $};
\draw[->] (c)--(d) node[midway,above] {$ \A \left( F(p), W \right)  $};
\draw[<->] (b)--(c) node[midway,right] {$ \cong  $} node[midway,left] {$ \chi _{A,W}  $};
\draw[<->] (a)--(d) node[midway,left] {$ \cong  $} node[midway,right] {$ \chi _{B,W}  $};
\draw[->] (a)--(b) node[midway,below] {$ \B \left( p, G(W) \right) $};
\end{tikzpicture}
\end{equation}
commutes.

\eqref{Lemma:lax-epi-2adjunction(a)} If $p: A\to B$ is a lax epimorphism in $\B $, for any $W\in \A$, we have that $\B \left( p, G(W) \right)$ is fully faithful and, hence, by the commutativity of \eqref{eq:diagram-chi-2adjunction}, $\A \left( F(p), W \right) $
is fully faithful.

\eqref{Lemma:lax-epi-2adjunction(b)} Assuming that $G$ is essentially surjective, if $F(p) : F(A) \to F(B) $ is a lax epimorphism in $\A $, then, for any $Z\in\B $, there is $W\in \A $ such that
$G(W) \cong Z$. Moreover,  $\A \left( F(p), W\right) $ is fully faithful and, hence, $\B\left( p, G(W)\right) $ is fully faithful by
the commutativity of \eqref{eq:diagram-chi-2adjunction}. This implies that $\B\left( p, Z\right) $ is fully faithful for any $Z\in\B $.
\end{proof}


\begin{definition}\label{d:2-nat-trans}
A $2$-natural transformation $\lambda : F\rightarrow G : \S\to\B $ is:
\begin{enumerate}[1.]
\item a \textit{pointwise lax epimorphism} if, for any $C\in\S $, the morphism $\lambda _ C : F(C)\to G(C) $ is a lax epimorphism in $\B $;
\item a \textit{lax epimorphism} if $\lambda $ is a lax epimorphism in the $2$-category
of $\tCat\left[ \S, \B  \right] $ of $2$-functors, $2$-natural transformations and modifications.
\end{enumerate}
\end{definition}

\begin{proposition}\label{p:pointwise-implies-global-laxepi}
Let $\lambda : F\rightarrow G : \S\to\B $  be a $2$-natural transformation.
If $\lambda $ is a pointwise lax epimorphism then it is a lax epimorphism in the $2$-category $\tCat\left[ \S, \B  \right] $.
\end{proposition}

\begin{proof}
Let $\lambda:F\to G:\A\to \B$ be a 2-natural transformation with each $\lambda_A:FA\to GA$ a lax epimorphism in $\B$. Let $\ga,\, \gb: G\to H: \A\to \B$ be two 2-natural transformations, and let $\Theta: \ga\ast\gl\leadsto \gb\ast \gl$ be a modification. In particular,  we have 2-cells in $\B$ indexed by $A\in \A$:
\begin{equation*}
\xymatrix@=1.5em{      &GA    \ar[rd]^{\ga_A}    \ar@{}[dd]^{\Theta_A}&\\
FA\ar[ru]^{\gl_A}   \ar[rd]_{\gl_A}&\Downarrow&HA \\
&GA\ar[ru]_{\gb_A}&}
\end{equation*}
This gives rise to unique 2-cells
\begin{equation*}
\xymatrix@=1em{ &&\ar@{=>}[dd]^{\Phi_A}&&  \\   GA    \ar@<1ex>@/^2pc/[rrrr]^{\ga_A}    \ar@<-1ex>@/_2pc/[rrrr]_{\gb_A}&&&&HA\\&&&&}
\end{equation*}
with $\Phi_A\ast \gl_A=\Theta_A$. The uniqueness of $\Phi=(\Phi_A)_{A\in \A}$ is clear. It is straightforward to see that $\Phi$ is indeed a modification.
\end{proof}

However, not every lax epimorphic $2$-natural transformation is a pointwise lax epimorphism. In fact, this is known to be true for epimorphisms and, as observed
in (\ref{lax-epi:epi}) of Examples \ref{es:lax-epi}, lax epimorphisms in locally discrete $2$-categories are the same as epimorphisms.

More precisely, consider the locally discrete $2$-category $\S $ generated by
\begin{equation*}
\begin{tikzpicture}[x=3cm, y=0.7cm]
\node (a) at (0,0) {$ A  $};
\node (b) at (1, 0) { $ B  $};
\node (c) at (2,0) {$ C $};
\node (c1) at (1.9,0.2) {};
\node (c2) at (1.9,-0.2) {};
\node (b1) at (1.1,0.2) {};
\node (b2) at (1.1,-0.2) {};
\draw[->] (a)--(b) node[midway,above] {$h  $};
\draw[->] (b1)--(c1) node[midway,above] {$ f $};
\draw[->] (b2)--(c2) node[midway,below] {$ g $};
\end{tikzpicture}
\end{equation*}
with the equation $fh = gh$.
The pair $(h,f)$ gives an epimorphism
in $\tCat\left[ \mathsf{2} , \S\right]$, where $\mathsf{2}$ is the category of two objects and a non-trivial morphism between them, but $h$ clearly is not an epimorphism in
$\S $. Since $\S$ and  $\tCat\left[ \mathsf{2} , \S\right]  $ are locally discrete,
this proves that $(h,f) $ gives a  $2$-natural transformation which is a lax epimorphism but not a pointwise lax epimorphism.

Yet, it follows from Lemma \ref{lem:2-adjunction-preservation-of-lax-epi} that
the converse holds for many interesting cases. More precisely:

\begin{theorem}
Let  $\B $ be a $2$-category with cotensor products. Then, a $2$-natural transformation
$\lambda : F\to G : \S\to \B $ is a lax epimorphism if and only if it is a pointwise lax epimorphism.
\end{theorem}

\begin{proof}
By Proposition \ref{p:pointwise-implies-global-laxepi},
every pointwise lax epimorphism is an epimorphism. We prove the converse below.

 Let $\mathsf{1}$ be the terminal category with only the object $0$.
  For each $s\in\S $, we denote by  $\functorS : \mathsf{1} \to \S $ the functor defined by $s$.
  For each $\functorB : \mathsf{1}\to\B  $, we have the pointwise right Kan extension (see \cite[Theorem~I.4.2]{Dubuc-KanExtensions}) given by
$$ \Ran _{\functorS } \functorB (a) = \lim\left( \S\left( a,\functorS - \right) , \functorB\right) \cong    \S\left( a, s \right)\pitchfork \left( \functorB 0\right).  $$
We conclude that,   for any $s\in\S$,  we have the $2$-adjunction
$$\tCat\left[ \functorS , \B  \right]\dashv \Ran _{\functorS  } . $$

  Therefore, by Lemma \ref{lem:2-adjunction-preservation-of-lax-epi},
 assuming that $\lambda : F\to G : \S\to \B $ is a lax epimorphism in $\tCat\left[ \S , \B  \right]  $, we have that, for every $s\in\S $,
$$\tCat\left[ \functorS , \B  \right]\left( \lambda \right) = \lambda \ast \id_{\functorS } = \lambda _s  $$
is a lax epimorphism in $\B $.
\end{proof}

\section{The orthogonal $LaxEpi$-factorization system}\label{(LE,LSM)}

Factorization systems in categories have largely shown their importance, taking the attention of many authors since the pioneering work exposed in \cite{FK}. (For a comprehensive account of the origins of the study of categorical factorization techniques see \cite{T1}.) When the category has appropriate colimits, we get one of the most common orthogonal factorization systems, the $(Epi, StrongMono)$ system. Since lax epimorphisms look like an adequate 2-version of epimorphisms, it is natural  to ask for a factorization system involving them.
In this section, we will obtain an orthogonal  $(LaxEpi, LaxStrongMono)$-factorization system in 2-categories.  In the next section we give a description of this orthogonal factorization system in $\Cat$.

The notion of orthogonal factorization system in 2-categories generalizes the ordinary one (see \cite{AHS} or \cite{HST}) by incorporating the two-dimensional aspect in the diagonal fill-in property. Here we use a strict version of the orthogonal factorization systems studied in   \cite{DV} (see Remark \ref{r:DV}):

\begin{definition}\label{d:ofs}
In the 2-category  $\A$, let $\ce$ and $\cm$ be two classes of morphisms closed under composition with isomorphisms from the left and the right, respectively.  The pair $(\ce, \cm)$ forms an {\em orthogonal factorization system} provided that:

\begin{itemize}
  \item[(i)] Every morphism $f$ of $\A$ factors as a composition $f=me$ with $e\in \ce$ and $m\in \cm$.
  \item[(ii)] For every $A\xrightarrow{e}B$ in $\ce$ and $C\xrightarrow{m}D$ in $\cm$, the square
  \begin{equation*}
    \xymatrix{\A(B,C)\ar[r]^{\A(B,m)}\ar[d]_{\A(e,C)}&\A(B,D)\ar[d]^{\A(e,D)}\\
    \A(A,C)\ar[r]^{\A(A,m)}&\A(A,D)}
  \end{equation*}is a pullback in $\Cat$.
\end{itemize}
\end{definition}

\begin{remark}\label{r:DV} In \cite{DV}, Dupont and Vitale studied orthogonal factorization systems in 2-categories  in a non-strict sense. Thus, in (i) of Definition \ref{d:ofs} the factorization holds up to equivalence, and in (ii), instead of a pullback, we have a bi-pullback.
\end{remark}

\begin{remark}\label{r:diag2}
(1) The one-dimensional aspect of (ii) asserts, for each pair of morphisms $f:A\to C$ and $g:B\to D$ with $mf=ge$, the existence of a unique $t:B\to C$ with $te=f$ and $mt=g$. The two-dimensional aspect of (ii) means that, whenever, with the above equalities, we have $t'e=f'$ and $mt'=g'$, and 2-cells $\alpha:f\to f'$ and $\beta:g\to g'$  such that  $m\ast \alpha =\beta\ast e$,
\begin{equation}\label{diag2}
\xymatrix{A\; \ar[rr]^e\ar@<-0.6ex>@/^0.8pc/[d]^{f'}\ar@<-0.7ex>@/_0.8pc/[d]_f^{\stackrel{\alpha}{\Rightarrow}}&&\;\, B\ar@<0.7ex>@/^0.8pc/[d]^g_{\stackrel{\beta}{\Leftarrow}}\ar@<0.6ex>@/_0.8pc/[d]_{g'}\ar@{-->}@/^0.8pc/[lld]^{t'}_{\stackrel{\theta}{\Rightarrow}}\ar@<-0.3ex>@{-->}@/_0.8pc/[lld]_t\\
C\; \, \ar@<-0.7ex>[rr]_m&&\;D}
\end{equation}
 then there is a unique 2-cell $\theta:t\to t'$ with $\theta \ast e=\alpha$ and $m\ast \theta =\beta$.

(2) If $\ce$ is made of lax epimorphisms,  the two-dimensional aspect comes for free. Indeed, for $\alpha:f=te\Rightarrow t'e=f'$, there is a unique $\theta:t\Rightarrow t'$ with $\theta\ast e=\alpha$; and, since $\beta\ast e=m\ast \alpha=m\ast \theta\ast e$, we have $\beta =m\ast \theta$.
\end{remark}

\begin{definition}\label{laxstrong}
A 1-cell $m:C\to D$  is said to be a {\em lax strong monomorphism} if it has the diagonal fill-in property \wrt lax epimorphisms; that is, for every commutative square
      \begin{equation}\label{square}\xymatrix{A\ar[r]^{e}\ar[d]_f&B\ar@{.>}[dl]^t\ar[d]^g\\
     C\ar[r]_m&D}
     \end{equation}
      with $e$ a lax epimorphism, there is a unique $t:B\to C$  such that  $te=f$ and $mt=g$.

      In other words, taking into account Remark \ref{r:diag2}(2),  $m:C\to D$  is  a  lax strong monomorphism  if for every lax epimorphism  $e$, the morphisms $e$ and $m$ fulfil condition (ii) of Definition \ref{d:ofs}.
\end{definition}

\begin{remark}
It is obvious that lax strong monomorphisms are closed under composition and left-cancellable; moreover, their intersection with lax epimorphisms are isomorphisms.
\end{remark}

\begin{proposition}\label{p:inserter} In a 2-category:
\begin{enumerate}
  \item[(i)]
Every inserter is a lax strong monomorphism.
\item[(ii)] In the presence of coequifiers, every lax strong monomorphism is faithful, i.e., a morphism $m$  such that  $\A(X,m)$ is faithful for all $X$.
 \end{enumerate}
\end{proposition}

\begin{proof}
(i) For the commutative square \eqref{square} above let $e$ be a lax epimorphism and let the diagram
\begin{equation}\label{inserter}
\xymatrix@=1em{      &D    \ar[rd]^r     \ar@{}[dd]^{\alpha}&\\
C\ar[ru]^m   \ar[rd]_m&\Downarrow&E \\
&D\ar[ru]_s&}
\end{equation}
be an inserter. Since $e$ is a lax epimorphism,   there is a unique $\beta:rg\Rightarrow sg$ with $\alpha\ast f=\beta\ast e$. This implies the existence of a unique $t:B\to C$  such that  $mt=g$ and $\alpha\ast t=\beta$. Then we have $\alpha \ast(te)=\beta\ast e=\alpha\ast f$ and $m(te)=ge=mf$. Hence, by the universality of $(m, \alpha)$, we conclude that $te=f$. And $t$ is unique: if $mt=mt'$ and $te=t'e$, then we have $\alpha\ast t\ast e=\alpha\ast t'\ast e$, which  implies $\alpha\ast t=\alpha\ast t'$; this together with $mt=mt'$ shows that $t=t'$.

(ii) Given a lax strong monomorphism $m:A\to B$ and two 2-cells $\alpha, \beta:r\to s:X\to A$ with $m\ast \alpha=m\ast \beta$, let $e:A\to C$ be the coequifier of the 2-cells. Then $m$ factors through $e$. Since, by \ref{es:lax-epi}(\ref{lax-epi:coeq}), $e$ is a lax epimorphism,  using the diagonal fill-in property, there is some $t:C\to A$ with $te=1_A$. Then $\alpha=\beta$.
\end{proof}

\begin{examples}
\begin{enumerate}
    \item
In $\Pos$ and $\Preord$ the converse of \ref{p:inserter}(i) also holds. In $\Pos$ lax strong monomorphisms are just order-embeddings\footnote{A morphism $m:X\to Y$ in $\Pos$ or $\Preord$ is an \textit{order-embedding} if $m$ is injective and $m(x)\leq m(y) \Leftrightarrow x\leq y$.} and order-embeddings coincide with inserters (\cite[Lemma 3.3]{ASV}).

Also in $\Preord$ lax strong monomorphisms coincide with inserters. It is easily seen that lax strong monomorphisms are precisely the order-embeddings  $m:X\to Y$ with $m[X]$ closed in $Y$ under isomorphic elements. Let $m:X\to Y$ be a lax strong monomorphism. Let $Z$ be obtained from $Y$ just replacing every element $y\in Y\setminus m[X]$ by two unrelated elements $(y,1)$ and $(y,2)$, and let the maps $f_1, f_2:Y\to Z$ be equal on $m[X]$ and $f_i(y)=(y,i)$, $i=1,2$, for the other cases. Endowing $Z$ with the least preorder which makes $f_1$ and $f_2$ monotone, we see that $m$ is the inserter of $f_1$ and $f_2$.
 \item But, in general, the converse of \ref{p:inserter}(i) is false. Just consider an ordinary category (i.e. a locally discrete 2-category)  with an orthogonal $(Epi, StrongMono)$-factorization system, where strong monomorphisms and regular monomorphisms do not coincide. This is the case, for instance, of the category of semigroups (see \cite[14I]{AHS}).
 In the 2-category $\Cat$ the coincidence of the inserters with the lax strong monomorphisms is left as an open problem (see Question \ref{quest1}).

\end{enumerate}
\end{examples}

\begin{remark}\label{counterex} In contrast to  \ref{p:inserter}, neither equifiers nor equalizers are, in general, lax strong monomorphisms.
Consider the following equivalence of categories, where only the non trivial morphisms are indicated:
\begin{center}
\setlength{\unitlength}{0.4mm}
\begin{picture}(120,40)
\put(-3,20){$A=$}
\put(19,20){$a$}
\put(15,14){\line(1,0){15}}
\put(15,14){\line(0,1){15}}
\put(30,29){\line(-1,0){15}}
\put(30,29){\line(0,-1){15}}
\put(38,20){\vector(1,0){20}}
\put(47,25){$E$}

\put(72,20){$a $}
\put(80,23){${\longrightarrow}$}
\put(85,28){\scriptsize $f$}
\put(96,20){$ b$}
\put(80,17){${\longleftarrow}$}
\put(84,12){\scriptsize $f^{-1}$}
\put(68,7){\line(1,0){37}}
\put(68,7){\line(0,1){30}}
\put(105,37){\line(-1,0){37}}
\put(105,37){\line(0,-1){30}}

\put(108,20){$=B$}
\end{picture}
\end{center}
The functor $E$ is a lax epimorphism  (see Example \ref{es:lax-epi}(\ref{lax-epi:equi})), but not a lax strong monomorphism,   since there is no $T:B\to A$ making the following two  triangles
$$\xymatrix@=1.4em{A\ar@{=}[d]\ar[r]^E&B\ar@{=}[d]\ar[ld]\ar@<-4pt>@{}[dl]^{T} \\ A\ar[r]_E&B}$$
commutative.
But  $E$ is both an equifier and an equalizer. To see that it is an equalizer consider the pair $F,\, \id_B:B\to B$, where $F$ takes all objects to $a$ and all morphisms to $1_a$. To see that it is an equifier consider the category
\begin{equation}\label{category}
C=\; \fbox{\begin{minipage}{14em}
    \xymatrix@=1.2em{Ra\ar[dd]_{\alpha_a=\beta_a}\ar@<0.4ex>@/^0.5pc/[rr]^{Rf}&&Rb\ar@<0.4ex>@/^0.5pc/[ll]^{Rf^{-1}}\ar@<2ex>@/^0.7pc/[dd]^{\alpha_b}\ar@<2ex>@/_0.7pc/[dd]_{\beta_b}\\ \\
    Sa\ar@<0.4ex>@/^0.5pc/[rr]^{Sf}&&Sb\ar@<0.4ex>@/^0.5pc/[ll]^{Sf^{-1}}
    }
\end{minipage}}
\end{equation}
and 2-cells $\alpha,\, \beta:R\to S:B\to C$ given in the obvious way.
\end{remark}

A key property in the sequel is the closedness of lax epimorphisms under colimits, in the sense of \ref{d:stcolim} below. The closedness of classes of morphisms under limits in ordinary categories was studied in \cite{IK}.

\begin{definition}\label{d:stcolim}
  Let $\ce$ be a class of morphisms in a 2-category $\B$. We say that {\em $\ce$ is closed under ($2$-dimensional) colimits} in $\B$ if, for every small 2-category  $\S $, every weight $W : \S ^\op \to\Cat $  and every 2-natural transformation $\lambda:D\to D':\S\to \B$, the induced morphism
   $$\colim \left( W, \lambda  \right) : \mathrm{colim}(W,D)\to \mathrm {colim}(W,D')$$
   is a morphism in the class $\ce$  whenever, for any $C\in\S $,  $\lambda _ C$ is a morphism in $\ce $.
\end{definition}

\begin{theorem}\label{t:stcolim} Lax epimorphisms are closed under ($2$-dimensional) colimits.
\end{theorem}
\begin{proof}
In fact, if the $2$-natural transformation $\lambda:D\to D':\S\to \B$ is a pointwise lax epimorphism, then, for any $A\in\B $, the $2$-natural transformation
$$\B \left( \lambda , A  \right) : \B \left(  D'- , A \right)\rightarrow\B \left(  D - , A \right),  $$
pointwise defined by
$\B \left( \lambda , A\right) _ C = \B \left( \lambda _ C, A \right) $,
is pointwise fully faithful. Hence it is fully faithful in the 2-category $\Cat[\S,\B]$ (dual of Proposition \ref{p:pointwise-implies-global-laxepi}).
Therefore, for any weight $W:\S^{\op}\to \Cat$ and $X\in\B $,
$$ \B\left( \colim \left( W, \lambda  \right), X  \right)\cong \tCat\left[ \S , \B \right]\left( W, \B \left( \lambda , A\right) \right)  $$
is fully faithful. This proves that $\colim \left( W, \lambda  \right)$ is a lax epimorphism in $\B $.
\end{proof}

\begin{remark}\label{r:closednessEM} As shown in \cite{bousfield}, see also \cite{FK}, for any  orthogonal $(\ce,\cm)$-factorization system in an ordinary category, $\ce$ and $\cm$ are closed under, respectively,  colimits and limits.

In $\Cat $, lax epimorphisms are not closed under ($2$-dimensional) limits, fully faithful functors are not closed under ($2$-dimensional) colimits, and, moreover, equivalences are neither closed under limits nor colimits.

Indeed, consider the category $\nabla \mathsf{2} $ with two objects and one isomorphism between them. Let $d^0$ and $d^1$ be the two possible inclusions $\mathsf{1}\to\nabla \mathsf{2} $ of
the terminal category in $\nabla\mathsf{2} $.
There is only one $2$-natural transformation $\iota $ between the diagram
\begin{equation*}
\begin{tikzpicture}[x=3cm, y=0.7cm]
\node (b) at (1, 0) { $ \mathsf{1}  $};
\node (c) at (2,0) {$ \nabla\mathsf{2} $};
\node (c1) at (1.9,0.2) {};
\node (c2) at (1.9,-0.2) {};
\node (b1) at (1.1,0.2) {};
\node (b2) at (1.1,-0.2) {};
\draw[->] (b1)--(c1) node[midway,above] {$ d^0 $};
\draw[->] (b2)--(c2) node[midway,below] {$ d^1 $};
\end{tikzpicture}
\end{equation*}
and the terminal diagram $\xymatrix{\mathsf{1} \ar@<+0.5ex>[r] \ar@<-0.5ex>[r] & \mathsf{1} .} $

Clearly, $\iota $ is a pointwise equivalence (and, hence, a pointwise lax epimorphism and fully faithful functor).
However, the induced functor between the equalizers and the coequalizers
are respectively
\begin{equation*}
    \overline{\iota} : \emptyset\to \mathsf{1} \qquad\mbox{and}\qquad \underline{\iota} :\Sigma\mathbb{Z}\to\mathsf{1}
\end{equation*}
in which $\Sigma\mathbb{Z}$ is  just the group $\left( \mathbb{Z} , + , 0\right) $ seen as a category with only one object. The functor $\overline{\iota}$ is not a lax epimorphism, while $\underline{\iota}$ is not fully faithful. Hence, neither functor is an equivalence.

Therefore, equivalences may not be the left or the right class of a (strict) orthogonal factorization system in a 2-category with reasonable (co)limits.
\end{remark}

The closedness under colimits has several nice consequences. We indicate three of them, which are going to be useful in the proof of Theorem \ref{t:ofs} below.

\begin{lemma}\label{l:push} (cf. \cite{IK}). Lax epimorphisms are stable under pushouts and cointersections. Moreover, the multiple coequalizer of a family of morphisms  equalized by a lax epimorphism  is  a lax epimorphism.
\end{lemma}
\begin{proof}
\begin{enumerate}
    \item
    \label{push1}
 Let the two squares in the following picture be pushouts:
$$\xymatrix@=1.2em{\bullet\ar@{=}[rrr]\ar[ddd]_g\ar@{.>}[rd]_{\id\!\!\!\!}&&&\bullet\ar[ddd]^g\ar@{.>}[ld]^{\!\!\!\!\!\!f}\\
&\bullet \ar[r]^f\ar[d]_g&\bullet \ar[d]^{g'}&\\
&\bullet \ar[r]_{f'}&\bullet&\\\bullet\ar@{.>}[ru]^{\id\!\!\!\!}\ar@{=}[rrr]&&&\bullet\ar@{-->}[lu]_{\!\!\!\!\!f'}}$$
Then, the dotted arrows form a 2-natural transformation between the corresponding origin  diagrams, and the dashed arrow is the unique one induced by the universality of the inner square. From Theorem \ref{t:stcolim}, if $f$  is a lax epimorphism, so is $f'$.
In conclusion, lax epimorphisms   are stable under pushouts.

\item
  \label{push2}
  Analogously, we see that the cointersection $e:A\to E$ of a family $e_i:A\to E_i$ of lax epimorphisms  is a lax epimorphism.  $$
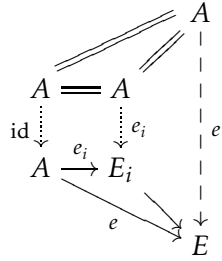
\xymatrix@=1.2em{&&A\ar@{-->}[ddd]^e\\A\ar@{=}[r]\ar@<1ex>@{=}[rru]\ar@{.>}[d]_{\id}&A\ar@{=}[ru]\ar@{.>}[d]^{e_i}&\\
A\ar[r]^{e_i}\ar[rrd]_e&E_i \ar[dr]&\\
&&E}$$

\item
  \label{push3}
 Let $f_i:B\to C$ be a family of morphisms  equalized by a lax epimorphism  $e$, i.e., $f_ie=f_je$ for all $f_i$ and $f_j$ of the family. Then, the closedness under colimits ensures that $c$ is a lax epimorphism, as illustrated by the diagram:
$$\xymatrix@=1.2em{E\ar[r]^{f_ie}\ar@{.>}[d]_e&B\ar@{=}[r]\ar@{.>}[d]^{\id}&B\ar@{-->}[d]^c\\
A\ar[r]_{f_i}&B\ar[r]_c&C}$$
\end{enumerate}
\end{proof}

\begin{remark}
Many of everyday categories are cowellpowered, that is, the family of epimorphisms with a same domain is essentially small.
By contrast, in the ``mother" of all 2-categories, $\Cat$, the class of lax epimorphisms  is not cowellpowered: For every cardinal $n$, let $A_n$ denote the category whose objects are $a_i$, $i\in n$, and whose morphisms are $f_{ij}:a_i\to a_j$ with $f_{jk}f_{ij}=f_{ik}$ and $f_{ii}=1_{a_i}$ for $i,j,k\in n$. Every  inclusion functor  $E_n:A_0\to A_n$, being an equivalence, is a lax epimorphism,   but the family of all these $E_n$ is a proper class. Moreover, the family $E_n$, $n\in \Card$, fails to have a cointersection in $\Cat$. However, $\Cat$ is almost cowellpowered in the sense of Definition \ref{d:almost} as shown in the next section.
\end{remark}

\begin{definition}\label{d:almost}
   Let $\ce$ be a class of 1-cells in a 2-category  $\A$. Given a morphism $f:A\to B$, denote by $\Ef$ the category whose objects     are factorizations $A\xrightarrow{d}D\xrightarrow{p}B$ of $f$ with $d\in \ce$, and whose morphisms $u:(d,D,p)\to (e,E,m)$ are 1-cells $u:D\to E$ with $ud=e$ and $md=p$. We say that  {\em $\A$ is almost cowellpowered \wrt $\ce$}, if $\Ef$  has a weakly terminal set  for every morphism $f$.
\end{definition}

The closedness of lax epimorphisms  under colimits allows to obtain the following theorem, whose proof makes use of a standard argumentation for the General Adjoint Functor Theorem.

\begin{theorem}\label{t:ofs} Let the 2-category  $\A$ have conical colimits and be almost cowellpowered \wrt lax epimorphisms.   Then $\A$ has an orthogonal $(LaxEpi, LaxStrongMono)$-factorization system.
\end{theorem}

\begin{proof} Let $\ce$ be the class of lax epimorphisms  in $\A$. Given a morphism $f:A\to B$, let $\{(e_i,E_i,m_i)\, |\, i\in I\}$ be a weakly terminal object    of the category  $\Ef$; that is, for every factorization $A\xrightarrow{d}D\xrightarrow{p}B$ of $f$ with $d\in \ce$ there is some $i$ and some morphism $u:(d,D,p)\to (e_i,E_i,m_i)$. Take  the cointersection $e:A\to E$ of all $e_i:A\to E_i$. By Lemma \ref{l:push}, the morphism $e$ belongs to $\ce$;
$$\xymatrix{A\ar[rd]_e\ar[r]^{e_i}&E_i\ar[d]^{t_i}\ar[r]^{m_i}&B\\
&E\ar[ru]_m&}$$
moreover, the cointersection gives rise to a unique $m:E\to B$ with $me=f$. Thus,  $(e,E,m)$ is clearly a weakly terminal object    of $\Ef$.

Consider all $s:E\to E$ forming a morphism $s:(e,E,m)\to (e,E,m)$ in $\Ef$. Let $c:E\to C$ be the multiple coequalizer of the family of all these morphisms $s:E\to E$. By Lemma \ref{l:push}, $c$ is a lax epimorphism.  Since $1_E$ is one of those morphisms $s$, and $ms=m$ for all of them, the universality of $c$ gives a unique  $n:C\to B$ with $nc=m$. It is easy to see that, $c:(e,E,m)\to (ce, C, n)$ is also  the coequalizer in $\Ef$ of all the above morphisms $s$. Since lax epimorphisms are closed under composition, $ce$ belongs to $\Ef$, hence,  $(ce, C, n)$ is  a terminal object     of $\Ef$ (cf. e.g. \cite{MacLane}, Ch.V, Sec.6).

We show that $n:C\to B$  is a lax strong monomorphism. In the following diagram, let the outer square be commutative with $q\in \ce$;  form the pushout $(\bar{q},\bar{r})$ of $q$ along $r$, and let $w$ be the unique morphism with $w\bar{q}=n$ and $w\bar{r}=s$:
\begin{equation}\label{squarep}\xymatrix@=1em{P\ar[rr]^{q}\ar[dd]_r&&Q\ar[ld]_{\bar{r}}\ar[dd]^s\\ &R\ar[rd]^w&\\
     C\ar[ru]^{\bar{q}}\ar[rr]_n&&B}
     \end{equation}
     The closedness under colimits of lax epimorphisms  ensures that $\bar{q}$ is a lax epimorphism (Lemma \ref{l:push}),   so $(\bar{q}
ce, R,w)\in \Ef$. Since $(ce,C,n)$ is terminal, there is a unique $u:R\to C$ forming a morphism in $\Ef$ from $(\bar{q}
ce, R,w)$ to $(ce,C,n)$, and it makes $u\bar{q}:C\to C$ an endomorphism on $(ce,C,n)$, then $u\bar{q}=1_C$. The morphism
$t=u\bar{r}$ fulfils the equalities $tq=r$ and $nt=s$. Moreover $t$ is unique; indeed, if  $t'$ is another morphism fulfilling the same equalities,  let $k$ be the  coequalizer of $t$ and $t'$ and let $p:K\to B$ be  such that  $pk=n$. Again by Lemma \ref{l:push}, $(kce,K,p)$ belongs to $\Ef$. Arguing as before for $\bar{q}$, we conclude that $k$ is a split monomorphism   and, then, $t=t'$.

Taking into account Remark \ref{r:diag2}, we conclude that we have indeed an orthogonal factorization system in the 2-category $\A$. \end{proof}

\begin{remark}
In \cite{DV}, an orthogonal factorization system $(\ce,\cm)$ which has $\ce$ made of lax epimorphisms  and $\cm$ made of faithful morphisms is said to be {\em (1,2)-proper}. By Proposition \ref{p:inserter}, this is the case for the $(LaxEpi, LaxStrongMono)$ factorization system. 
\end{remark}

\begin{examples}
Some of the well-known orthogonal factorization systems in ordinary categories are indeed of the $(LaxEpi, LaxStrongMono)$ type for convenient 2-cells. This is the case in the 2-categories
$\Pos$ and $\Grp$. In $\Pos$ it is the usual orthogonal \textit{(Surjections, Order-embeddings)}-factorization system.  Analogously for the category $\TTop_0$ of $T_0$-topological spaces and continuous maps, with 2-cells given by the pointwise  specialization order, we obtain \textit{(Surjections, Embeddings)}.
For the 2-category $\Grp$, the $(LaxEpi, LaxStrongMono)$ factorization is precisely the $(RegEpi,Mono)$ one.
\end{examples}

Recall that, for every category with an orthogonal factorization system $(\ce,\cm)$, we have that $\cm=\ce^{\downarrow}$, i.e., $\cm$ consists of all morphisms $m$ fulfilling the diagonal fill-in property as in \eqref{square} of Definition \ref{laxstrong}. From the proof of Theorem \ref{t:ofs},  it immediately follows  that, more generally, we have the following:

\begin{corollary}
Let $\ce$ be a class of morphisms closed under  post-composition with  isomorphisms in a cocomplete  category $\ca$. Then, $(\ce, \ce^{\downarrow})$ forms an orthogonal factorization system if and only if $\ca$ is almost cowellpowered with respect to $\ce$ and $\ce$ is closed under composition and under colimits.
\end{corollary}

\begin{proof}
Following the proof of Theorem \ref{t:ofs}, we see that, if $\ce$ is a class of morphisms closed under composition and under colimits, and  $\ca$ is almost cowellpowered with respect to $\ce$, then  $(\ce, \ce^{\downarrow})$ forms an orthogonal factorization system. Conversely,  a category with an orthogonal $(\ce,\cm)$-factorization system is almost cowellpowered with respect to $\ce$: the $(\ce,\cm)$-factorization of $f:A\to B$ is indeed a terminal object of $\Ef$. The closedness of $\ce$ under composition and under colimits  is a well-known fact for orthogonal factorization systems.
\end{proof}

\section{The Lax Epi-factorization  in $\Cat$}

We describe the orthogonal $(LaxEpi,LaxStrongMono)$-factorization system  in the 2-category $\Cat$ of small categories, functors and natural transformations.  Everything  we do in this section applies also to the bigger universe $\CAT$ of possibly large (locally small) categories.

Let us recall, by the way,  two
 well-known orthogonal factorization systems $(\ce,\cm)$ in the category $\Cat$:

\begin{enumerate}
\item[(a)] $\ce$  consists of all functors bijective on objects and $\cm$ consists of all fully faithful functors.

\item[(b)] $\ce$ consists of all initial functors and $\cm$ consists of all discrete opfibrations; analogously, for final functors and discrete fibrations~\cite{SW}.
\end{enumerate}

It is easy to see that in both cases, (a) and (b), the system $(\ce,\cm)$ fulfils the two-dimensional aspect of the fill-in diagonal property, thus we have an orthogonal factorization system in the 2-category $\Cat$ as defined in \ref{d:ofs}.



\vskip3mm

We recall a characterization of the lax epimorphisms
in the 2-category $\Cat$ of small categories, functors and natural transformations presented in \cite{AeBSV}.

Given a functor $F:A\to B$ and a morphism $g:b\to c$ in $B$, let $$g\dd F$$ denote the category whose objects are triples $(h,a,k)$ such that  the composition $b\xrightarrow{h}Fa\xrightarrow{k}c$ is equal to $g$, and whose morphisms $f:(h,a,k)\to (h',a',k')$ are those $f:a\to a'$ of $A$ with $Fa\cdot h=h'$ and $k'\cdot Fa=k$. Then:

\begin{theorem}\label{t:absv} {\em \cite{AeBSV}} A functor $F:A\to B$  is a lax epimorphism in $\Cat $ if and only if, for every morphism $g$ of $B$, the category $g\dd F$ is connected.
\end{theorem}
\vskip3mm

We start by defining discrete splitting bifibrations. We will see that they are precisely the lax strong monomorphisms.

\begin{notation}\label{n:t.C}For a  functor $P:A\to B$  and every decomposition of a morphism $g$ of the form $b\xrightarrow{r}Pe\xrightarrow{s}c$, we denote by $[(r,s)]$ the corresponding connected component in the category $g\dd P$.  By composing  a morphism $t:d\to b$ with $C=[(r,s)]$ we obtain $C\cdot t=[(rt,s)]$, a connected component of $tg\dd P$. Analogously, for the composition on the
 right hand side: for $u:b\to c$, $u\cdot C=[(h,uk)]$.
 \end{notation}

 \begin{definition}\label{d:discrete} Let $P:E\to B$ be a functor.
 \begin{enumerate}
     \item[(a)] A {\em $P$-split} consists of a factorization of an identity $1_b$ of the form
 $$\xymatrix{b\ar[r]^h\ar@/^-1.5pc/[rr]_{1_b}&Pe\ar[r]^k&b}$$
 with $[(1_{Pe},hk)]=[(hk,1_{Pe})]$.

\vskip1mm

 \item[(b)] A {\em $P$-split diagram} is a rectangle
 \begin{equation}\label{P-split}
     \xymatrix{b\ar[r]^h\ar[d]_g&Pe\ar[r]^k\ar@{~}[d]&b\ar[d]^g\\
 c\ar[r]_{h'}&Pe'\ar[r]_{k'}&c}
 \end{equation}
 where  $(h,k)$ and $(h',k')$ are $P$-splits  such that  $[(h,gk)]=[(h'g,k')]$ in $g\dd P$. The wavy line in the middle of the rectangle indicates the existence of an appropriate $P$-zig-zag between $(h,gk)$ and $(h'g,k')$; that is, the existence of a finite number of morphisms $h_i$, $k_i$, $f_i$  making   the following diagram commutative:
 $$\xymatrix{b\ar[r]^h\ar@{=}[d]&Pe\ar[d]^{Pf_0}\ar[r]^k&b\ar[d]^g\\
 b\ar[r]^{h_1}\ar@{=}[d]&Pe_1\ar[r]^{k_1}&c\ar@{=}[d]\\
 b\ar[r]^{h_2}\ar@{=}[d]&Pe_2\ar[u]_{Pf_1}\ar[d]^{Pf_2}\ar[r]^{k_2}&c\ar@{=}[d]\\
 b\ar[r]^{h_3}\ar@{=}[d]&Pe_3\ar@{.}[d]\ar[r]^{k_3}&c\ar@{=}[d]\\
 b\ar[r]^{h_n}\ar[d]_g&Pe_n\ar@{.}[d]\ar[r]^{k_n}&c\ar@{=}[d]\\
 c\ar[r]_{h'}&Pe'\ar[u]_{Pf_n}\ar[r]_{k'}&c}$$

\vskip1mm

 \item[(c)] The functor $P:E\to B$ is said to be a {\em discrete splitting bifibration} if, for every $P$-split diagram  \eqref{P-split}, there is a unique commutative rectangle in $E$ of the form
 \begin{equation}\label{split}
     \xymatrix{b_0\ar[r]^{h_0}\ar[d]_{g_0}&e\ar[r]^{k_0}&b_0\ar[d]^{g_0}\\
 c_0\ar[r]_{h'_0}&e'\ar[r]_{k'_0}&c_0}
 \end{equation}
 whose image by $P$ is the outer rectangle of \eqref{P-split}. (That is, $Px_0=x$, for each letter $x$ with $x_0$ appearing in \eqref{split}.)
  \end{enumerate}
 \end{definition}

 \begin{remark}\label{r:unique}
 If $P$ is a discrete splitting bifibration, then it is clear that, for every $P$-split of $1_b$,
 $$b\xrightarrow{h}Pe\xrightarrow{k}b$$
 there are unique morphisms $h_0:b_0\to e$ and $k_0:e\to b_0$  such that  $Ph_0=h$ and $Pk_0=k$.
 \end{remark}

 \begin{proposition}\label{p:faithful}
 Every discrete splitting bifibration
 \begin{enumerate}
     \item is faithful,
     \item is conservative, and
     \item reflects identities.
 \end{enumerate}
 \end{proposition}

 \begin{proof}
 Let $P:E\to B$ be a discrete splitting bifibration.
 \begin{enumerate}
     \item For    $\xymatrix{a\ar@<0.3pc>[r]^f\ar@<-0.3pc>[r]_g&b}$ with $Pf=Pg=x$,  consider the    following  diagrams:
     $$\xymatrix{Pa\ar@{=}[r]\ar[d]_x&Pa\ar@{~>}[d]_{Pf}\ar@{=}[r]&Pa\ar[d]^x\\
    Pb\ar@{=}[r]&Pb\ar@{=}[r]&Pb}\qquad \quad \xymatrix{a\ar@{=}[r]\ar[d]_f&a\ar@{=}[r]&a\ar[d]^f\\
    b\ar@{=}[r]&b\ar@{=}[r]&b}\qquad \quad \xymatrix{a\ar@{=}[r]\ar[d]_g&a\ar@{=}[r]&a\ar[d]^g\\
    b\ar@{=}[r]&b\ar@{=}[r]&b}$$
    The first one is a $P$-split rectangle and it is the image by $P$ of the two last ones.  Then $f=g$.
 \item Let $f:a\to b$ be  such that  $Pf$ is an isomorphism in $B$. Then we have a $P$-split diagram:
 $$\xymatrix{Pb\ar[d]_{(Pf)^{-1}}\ar@{=}[r]&Pb\ar@{=}[r]&Pb\ar[d]^{(Pf)^{-1}}\\
    Pa\ar@{=}[r]&Pa\ar@{~>}[u]_{Pf}\ar@{=}[r]&Pa}$$
 Consequently, there is a unique $t_0:b\to a$ with $Pt_0=(Pf)^{-1}$. Since, by (1),  $P$ is faithful, $t_0$ is the inverse of $f$.
 \item Let $f:d\to e$ be  such that  $Pf=1_x$. By (2), $f$ is an isomorphism. Concerning the diagrams $$\xymatrix{x\ar@{=}[r]\ar@{=}[d]&Pe\ar@{~>}[d]^{P1_e}\ar@{=}[r]&x\ar@{=}[d]\\
    x\ar@{=}[r]&Pe\ar@{=}[r]&x}\qquad \quad \xymatrix{d\ar[r]^f\ar[d]_{1_d}&e\ar[r]^{f^{-1}}&d\ar[d]^{1_d}\\
    d\ar[r]^f&e\ar[r]^{f^{-1}}&d}\qquad \quad \xymatrix{d\ar[r]^f\ar[d]_{f}&e\ar[r]^{f^{-1}}&d\ar[d]^{f}\\
    e\ar[r]^{1_e}&e\ar[r]^{1_e}&e}$$
    the first one is a $P$-split rectangle which is the image by $P$ of the two rectangles on the right hand side. Consequently, $f=1_d$.
 \end{enumerate}
 \end{proof}

 \begin{theorem}\label{t:f-Cat}
 For $\ce$ the class of lax epimorphisms and $\cm$ the class of discrete splitting bifibrations,  $(\ce,\cm)$  is an orthogonal factorization system in $\Cat$ (and also in $\CAT$).
 \end{theorem}

 \begin{proof} Along the proof we represent the categories by blackboard bold letters: $\A$, $\B$, etc.

 \vskip2mm

 \noindent (1) The factorization. Given a functor $F:\A\to \B$, we define the category  $\E$ as follows:

 \begin{itemize}
     \item[]
$\text{ob}\,\E$: pairs $(b,B)$ where $b\in  \B$ and $B$ is a connected component of the category  $1_b\dd F$, for which some representative is an $F$-split;
\item[] $\text{mor}\,\E$: all $(b,B)\xrightarrow{g}(c,C)$ with $g:b\to c$ a morphism of $\B$ and $g\cdot B=C\cdot g$, see Notation \ref{n:t.C}.

\item[] The identities and composition are obvious.
 \end{itemize}

Let $$\E\xrightarrow{P}\B$$
be the obvious projection, and define
$$\A\xrightarrow{E}\E$$
by
$Ea=(Fa, C_a)$ where $C_a$ is the connected component of $(1_{Fa}, 1_{Fa})$ in $1_{Fa}\dd F$, and $E(a\xrightarrow{f}a')=((Fa, C_a)\xrightarrow{Ff}(Fa', C_{a'}))$.  $\E$ is clearly well-defined and $F=P\cdot E$.

 \vskip2mm

 \noindent
(2) $E$ is a lax epimorphism.  We need to show that, for every $(b,B)\xrightarrow{g}(d,D)$ in $\E$, the category  $g\dd E$ is connected.

Every identity $(b,B)\xrightarrow{1_b}(b,B)$ factorizes through $Ea$ for some $a\in \A$. Indeed, $B$ contains some $F$-split $b\xrightarrow{h}Fa\xrightarrow{k}b$, and this means that $(b,B)\xrightarrow{h}Ea$ and $Ea\xrightarrow{k} (b,B)$ are morphisms in $\E$. From this, it immediately follows that the category $g\dd E$ is nonempty for all morphism $g$ in $\E$.

 Given two factorizations $(u_i,Ea_i,v_i)$, $i=1,2$, of $g$  in $\E$ as in the figure
\begin{equation}\label{E3}
\xymatrix{&(Fa_1,C_{a_1})\ar[rd]^{v_1}&\\
(b,B)\ar[rr]^g\ar[rd]_{u_2}\ar[ru]^{u_1}&&(d,D)\\
&(Fa_2,C_{a_2})\ar[ru]_{v_2}&}
\end{equation}
by the definition of morphisms in $\E$, we have the following equalities of connected components in $g\dd F$ (see Notation \ref{n:t.C}): $g\cdot  B=v_1u_1\cdot B=v_1\cdot C_{a_1}\cdot u_1=v_1\cdot[(u_1,1_{Fa_1})]=[(u_1,v_1)]$; and, analogously, $g\cdot B=[(u_2,v_2)]$, showing that $[(u_1,v_1)]=[(u_2,v_2)]$ in $g\dd F$; hence, $[(u_1,v_1)]=[(u_2,v_2)]$ also in $g\dd E$.

 \vskip2mm

 \noindent
(3) $P$  is a discrete splitting bifibration.

(3a) First observe that, given two factorizations in $\B$ of a same morphism $g$ of the form
$$\xymatrix@=1.5em{&P(e,E)\ar[rd]^s&\\
b\ar[ru]^{r}\ar[rd]_{r'}&&c\\
&P(e',E')\ar[ru]_{s'}}$$
if $\big(r,(e,E),s\big)$ and $\big(r', (e', E'), s'\big)$ belong to the same connected component of $g\dd P$, then  $s\cdot E\cdot r=s'\cdot E'\cdot r'$ in $g\dd F$. Indeed, a $P$-zig-zag  connecting the two factorizations, as illustrated in the left hand side diagram below, gives rise to an $F$-zig-zag connecting $s\cdot E\cdot r$ to $s'\cdot E'\cdot r'$ in $g\dd F$, as indicated in the right hand side diagram, where $E=[(h,a,k)]$, $E'=[(h',a',k')]$ and $E_j=[(h_j,a_j,k_j)]$:
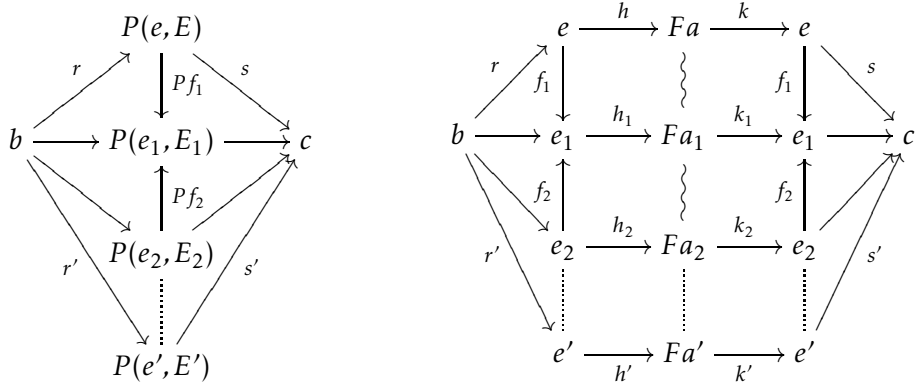
$$\xymatrix{&P(e,E)\ar[d]^{Pf_1}\ar[rd]^s&\\
b\ar[ru]^r\ar[r]\ar[rd]\ar[rdd]_{r'}&P(e_1, E_1)\ar[r]&c\\
&P(e_2, E_2)\ar[u]_{Pf_2}\ar[ru]\ar@{.}[d]&\\
&P(e',E')\ar[ruu]_{s'}&}\qquad \qquad \xymatrix{&e
\ar[d]_{f_1}\ar[r]^h&Fa\ar@{~}[d]\ar[r]^k&e\ar[d]_{f_1}\ar[rd]^s&\\
b\ar[ru]^r\ar[r]\ar[rd]\ar[rdd]_{r'}&e_1\ar[r]^{h_1}&Fa_1
\ar[r]^{k_1}&e_1\ar[r]&c\\
&e_2\ar@{.}[d]\ar[r]^{h_2}\ar[u]^{f_2}&Fa_2\ar@{~}[u]\ar[r]^{k_2}\ar@{.}[d]&e_2\ar@{.}[d]\ar[ru]\ar[u]^{f_2}&\\
&e'\ar[r]_{h'}&Fa'\ar[r]_{k'}&e'\ar[ruu]_{s'
}&}$$

(3b) Let
\begin{equation}\label{D4}
    \xymatrix{b\ar[d]_g\ar[r]^{u_1}&P(d,D)\ar@{~}[d]\ar[r]^{v_1}&b\ar[d]^g\\
    c\ar[r]_{u_2}&P(e,E)\ar[r]_{v_2}&c}
\end{equation}
be a $P$-split diagram with $D=[(h_1, a_1, k_1)]$ and $E=[(h_2, a_2, k_2)]$, where $(h_i,a_i,k_i)$ is  $F$-split, $i=1,2$. Let $B$ and $C$ be the connected components of $1_b\dd F$ and $1_c\dd F$ given, respectively, by
$$B=v_1\cdot D\cdot u_1=[(h_1u_1, a_1, v_1k_1)]\quad \mbox{ and } \quad C=v_2\cdot E\cdot u_2=[(h_2u_2, a_2, v_2k_2)].$$
 By (3a), since $[(1_d, u_1v_1)]=[(u_1v_1, 1_d)]$ in $u_1v_1\dd P$, we have that $u_1v_1\cdot D=D\cdot u_1v_1$. Then $u_1B=u_1v_1Du_1=Du_1v_1u_1=Du_1$ and  $Bv_1=v_1Du_1v_1=v_1u_1v_1D=v_1D$. Thus, in order to conclude that $(b,B)\xrightarrow{u_1}(d,D)$ and $(d,D)\xrightarrow{v_1}(b,B)$ are morphisms in $\E$, we just need to show that $(b,B)$ is an object of $\E$. Indeed, from the equalities $[(h,uvk)]=[(huv, k)]$ and $[1_{Fa}, hk)]=[(hk,1_{Fa})]$, where the subscripts of the letters were removed by the sake of simplicity, we see that the pair $(hu, a, vk)$ is an $F$-split:

 $\begin{array}{ll}[(1_{Fa},huvk)]&
 =huvk\cdot[(1_{Fa},hk)]=huvk\cdot[(hk, 1_{Fa})]=h\cdot[(h,uvk)]\cdot k=h\cdot[(huv,k)]\cdot k\\
 &=[(huvk,hk)]=[(1,hk)]\cdot huvk=[(hk,1)]\cdot huvk
 =[(huvk,1_{Fa})].
 \end{array}$\newline
 And $B$ is unique, because, if $B'$ is a connected component of $1_b\dd F$  such that
$u_1B'=Du_1$ and $v_1D=B'v_1$, then $B'=v_1u_1B'=v_1Du_1=B$. Analogously for $c\xrightarrow{u_2}P(e,E)\xrightarrow{v_2}c$.

It remains to be shown that $g:(b,B)\to (c,C)$ is a morphism of $\E$. By (3a), the $P$-split diagram \eqref{D4} gives rise to a  diagram of the form
$$\xymatrix{b\ar[d]_g\ar[r]^{u_1}&d\ar[r]^{h_1}&Fa_1\ar@{~}[d]\ar[r]^{k_1}&d\ar[r]^{v_1}&b\ar[d]^g\\
c\ar[r]_{u_2}&e\ar[r]_{h_2}&Fa_2\ar[r]_{k_2}&e\ar[r]_{v_2}&c}$$
showing that
$$gv_1Du_1=v_2Eu_2g \; \; \mbox{ in } \; g\dd F.$$
Hence, by definition of $B$ and $C$,
$$gB=Cg,$$
that is, $g$ is a morphism in $\E$. Since $(b,B)$ and $(c,C)$ are unique, $g$ is clearly unique too. In conclusion, we have a unique diagram of morphisms of $\E$ of the form
\begin{equation*}\label{P4}
    \xymatrix{(b,B)\ar[d]_g\ar[r]^{u_1}&(d,D)\ar[r]^{v_1}&(b,B)\ar[d]^g\\
    (c,C)\ar[r]_{u_2}&(e,E)\ar[r]_{v_2}&(c,C)}
\end{equation*}
whose image by $P$ is the rectangle of \eqref{D4}.

 \vskip2mm

 \noindent
(4) $(\ce,\cm)$ fulfils the diagonal fill-in property.
Let
\begin{equation*}
    \xymatrix{\A\ar[r]^Q\ar[d]_G&\B\ar[d]^H\\
    \C\ar[r]_M&\D}
\end{equation*}
be a commutative diagram where $Q$ is a lax epimorphism and $M$ is a discrete splitting bifibration.

(4a)  We define $T:\B\to \C$ as follows:

Given $b\in \B$, since $Q$ is a lax epimorphism,   the category  $b\dd Q$ is connected. Let $B$ be the unique connected component of $1_b\dd Q$, and let $(h,a,k)$ be a representative of $B$. It is a $Q$-split, since $(1_{Qa},hk)$ and $(hk, 1_{Qa})$ belong to the same connected component of $hk\dd Q$. Hence,
$$Hb\xrightarrow{Hh}MGa\xrightarrow{Hk}Hb$$
is an $M$-split in $\D$.

By Remark \ref{r:unique}, since $M$ is a discrete splitting bifibration, there are unique morphisms $h_0:b_0\to Ga$ and $k_0:Ga\to b_0$ with $Mh_0=Hh$ and $Mk_0=Hk$. We put
\begin{equation}\label{E6}Tb=b_0.
\end{equation}

We show that $b_0$ does not depend on the representative of $B$. Indeed, for another representative $(h',a',k')$, we have a $Q$-split diagram as on the left hand side of \eqref{E7};  by  applying $H$, we get  the $M$-split diagram on the right hand side:
\begin{equation}\label{E7}
  \xymatrix{b\ar@{=}[d]\ar[r]^h&Qa\ar@{~}[d]\ar[r]^k&b\ar@{=}[d]\\
  b\ar[r]_{h'}&Qa'\ar[r]_{k'}&b}\qquad \qquad \xymatrix{Hb\ar@{=}[d]\ar[r]^{Hh}&MGa\ar@{~}[d]\ar[r]^{Hk}&Hb\ar@{=}[d]\\
    Hb\ar[r]_{Hh'}&MGa'\ar[r]_{Hk'}&Hb}
\end{equation}
By hypothesis, there is a unique diagram
$$\xymatrix{b_1\ar[d]_{s}\ar[r]^{h_1}&Ga\ar[r]^{k_1}&b_1\ar[d]^s\\
  b_2\ar[r]_{h_2}&Ga'\ar[r]_{k_2}&b_2}$$
whose image by $M$ is the outside rectangle of the first diagram  of \eqref{E7}. But $M$ reflects identities, by Proposition \ref{p:faithful}. Then $s$ is an identity and, taking into account the unicity of $b_0$ and $k_0$ above, it must be $b_1=b_2=b_0$ and $s=1_{b_0}$.

Let
$$b\xrightarrow{g}c$$
be a morphism in $\B$. Since $Q$ is a lax epimorphism,   there is some $Q$-split diagram of the form
\begin{equation*}
    \xymatrix{b\ar[d]_g\ar[r]^{h_1}&Qa_1\ar@{~}[d]\ar[r]^{k_1}&b\ar[d]^g\\
    c\ar[r]^{h_2}&Qa_2\ar[r]^{k_2}&c}\, .
\end{equation*}
By applying $H$ to it, we obtain an $M$-split diagram:
\begin{equation}\label{E8}
    \xymatrix{Hb\ar[d]_{Hg}\ar[r]^{Hh_1}&MGa_1\ar@{~}[d]\ar[r]^{Hk_1}&Hb\ar[d]^{Hg}\\
    Hc\ar[r]^{Hh_2}&MGa_2\ar[r]^{Hk_2}&Hc}\, .
\end{equation}
By hypothesis, there are unique morphisms
\begin{equation}\label{E8a}
    \xymatrix{b_0\ar[d]_{g_0}\ar[r]^{\hat{h}_1}&Ga_1\ar[r]^{\hat{k}_1}&b\ar[d]^{g_0}\\
    c_0\ar[r]^{\hat{h}_2}&Ga_2\ar[r]^{\hat{k}_2}&c_0}
\end{equation}
making the diagram commutative and whose image by $M$ is the rectangle of \eqref{E8}.
We put
$$Tg=g_0.$$
Again, by the unicity, we know that $b_0$ and $c_0$ do not depend on the representative of $1_b\dd Q$ and $1_c\dd Q$. And the unicity of $g_0$ follows then from the faithfulness of $M$ (Proposition \ref{p:faithful}).

 $T$ is clearly a functor, the preservation of identities and composition being obvious.

(4b) We show that $T$ satisfies the diagonal fill-in condition.

Given $b\in \mbox{ob} \B$, $MTb=Mb_0=Hb$, by construction, and, analogously, $MTg=Hg$, for each $g\in \mbox{mor}\B$.

Given $f:a\to a'$ in $\A$, the $M$-split diagram
\begin{equation*}
    \xymatrix{HQa=MGa\ar[d]_{HQf=MGf}\ar@{=}[r]&MGa\ar@{.>}[d]^{MGf}\ar@{=}[r]&MGa\ar[d]^{MGf}\\
    HQa'=MGa'\ar@{=}[r]&MGa'\ar@{=}[r]&MGa'}
\end{equation*}
ensures that $TQf=Gf$.

Finally, if $T':\B\to \C$ is another functor  such that  $T'Q=G$ and $MT'=H$, we show that $T=T'$. Let $g:b\to d$ be a morphism of $\B$. The morphism $T(b\xrightarrow{g}d)=b_0\xrightarrow{g_0}d_0$ is the unique one making part of a commutative rectangle as in \eqref{E8a} whose image by $M$ is the rectangle of the $M$-split diagram \eqref{E8}. But the image by $M$ of the rectangle
$$\xymatrix{T'b\ar[d]_{T'g}\ar[r]^{T'h}\ar[d]&Ga\ar[r]^{T'k}&T'b\ar[d]^{T'g}\\
T'd\ar[r]_{T'h'}&Ga'\ar[r]_{T'k'}&T'd}$$
 gives also the $M$-split diagram  \eqref{E8}. Then $T'g=g_0=Tg$.
\end{proof}

\begin{question}\label{quest1}Inserters in $\Cat$ are discrete splitting bifibrations (by Proposition  \ref{p:inserter}). We don't know if the converse is true or not.
\end{question}

\section{Lax epimorphisms in the enriched context}

In this section we  study   lax epimorphisms in the enriched setting.

\begin{assumption}\label{assumption1} Throughout the section,
 $\cv=\left( \cv_0,\otimes, I\right)$ is a symmetric monoidal closed category with $\cv_0$ complete.
\end{assumption}

 We denote by $\cv$-$\Cat$ the 2-category of \textit{small} $\cv$-categories, $\cv$-functors and $\cv$-natural transformations.

Let $\ca $ be a small $\cv$-category, and $\cb $ a (possibly large) $\cv$-category. By abuse of language, we also denote by $\cv$-$\Cat(\ca,\cb)$ the (ordinary) category of $\cv$-functors from $\ca$ to $\cb$ and $\cv$-natural transformations between them.
Moreover, in this setting, the designation $\cv$-$\Cat[\ca,\cb]$  (or just $[\ca,\cb]$) represents the $\cv$-category of $\cv$-functors; thus, for any pair of $\cv$-functors $F, G : \ca\to\cb$,  the hom-object $\cv$-$\Cat[\ca,\cb](F,G)$ is given by  the end
$$ \int _{A\in \ca } \cb (F A , G A ) \, .$$

Recall that a $\cv$-functor $P:\ca\to\cb$ is {\em $\cv$-fully faithful} (called just {\em fully faithful} in \cite{kelly82}) if  the map $P_{A,A'}:\ca(A,A') \to \cb(PA,PA')$ is an isomorphism in $\cv_0$ for all $A,A'\in \ca$.

Let $\ci$ be the {\em unit $\cv$-category} with one object    $0$ and $\ci(0,0)=I$.
Given a $\cv$-functor $P:\ca\to\cb$,   the underlying functor of $P$ is denoted by
$P_0 = \cv\textrm{-}\Cat (\ci , P) :\ca_0\to\cb_0$.

In general, we use the notations of \cite{kelly82}; concerning limits, we denote a weighted limit over a functor $F:\cd\to \cc$ with respect to a weight $W:\cd \to \cv$ by $\lim(W,F)$ (called indexed limit and designated  by  $\{W,F\}$ in \cite{kelly82}).

\begin{lemma}\label{l:ff-enr}
For a $\cv$-functor $P:\ca\to\cb$, consider the following conditions.
\begin{enumerate}[(a)]

\item\label{lemma(a):characterization-ff-vfunctor}
$P$ is $\cv$-fully faithful.

\item\label{lemma(b):characterization-ff-vfunctor}
$P_0$ is fully faithful.

\item\label{lemma(c):characterization-ff-vfunctor}
The functor $\Cat(\cc,P_0):\Cat(\cc, \ca_0)\to \Cat(\cc,\cb_0)$ is fully faithful  for every (ordinary) category $\cc$.

\item\label{lemma(d):characterization-ff-vfunctor}
The functor $\cv$-$\Cat(\cc,P):\cv$-$\Cat(\cc, \ca)\to \cv$-$\Cat(\cc,\cb)$ is fully faithful  for every $\cv$-category $\cc$.

\item\label{lemma(f):characterization-ff-vfunctor}
The $\cv$-functor $\cv$-$\Cat[\cc,P]:\cv$-$\Cat[\cc, \ca]\to \cv$-$\Cat[\cc,\cb]$ is $\cv$-fully faithful for every $\cv$-category $\cc$.

\end{enumerate}
We have that
$$\xymatrix{\eqref{lemma(a):characterization-ff-vfunctor}\ar@{<=>}[r] &\eqref{lemma(f):characterization-ff-vfunctor}\ar@{=>}[r]&\eqref{lemma(d):characterization-ff-vfunctor}\ar@{=>}[r]
& \eqref{lemma(c):characterization-ff-vfunctor}\ar@{<=>}[r]
&\eqref{lemma(b):characterization-ff-vfunctor}}$$
 The five conditions are equivalent whenever (i) $P$ has a left or right $\cv$-adjoint, or (ii)  $V=\cv_0(I, -):\cv_0\to \Set$ is conservative.
\end{lemma}

\begin{proof}
It is well-known that \eqref{lemma(a):characterization-ff-vfunctor} $\Leftrightarrow$ \eqref{lemma(b):characterization-ff-vfunctor} in case we have (i) or (ii)~\cite[1.3 and 1.11]{kelly82}.

\eqref{lemma(b):characterization-ff-vfunctor}  $\Leftrightarrow$ \eqref{lemma(c):characterization-ff-vfunctor}. It is just Remark \ref{r:ff}.

\eqref{lemma(a):characterization-ff-vfunctor}  $\Rightarrow$ \eqref{lemma(d):characterization-ff-vfunctor}. Given two $\cv$-functors $F,G:\cc\to \ca$, and a $\cv$-natural transformation $\gb:PF\to PG$, we want to show that there is a unique $\cv$-natural transformation $\ga: F\to G$ with $P\ga=\gb$. Since $P$ is $\cv$-fully faithful, $P_{A,B}$ is a $\cv _0$-isomorphism for all $A, B\in \ca$. We just define $\ga:F\to G$ with each component $\ga_C$ given by
$$\ga_C\equiv \big(I\xrightarrow{\gb_C}\cb(PFC,PGC)\xrightarrow{(P_{FC,GC})^{-1}}\ca(FC,GC)\big).$$
Clearly $\beta_C=P\ga_C$ for each $C$, and $\ga$ is unique. From the $\cv$-naturality of $\gb$ and the fact that $P$ is a $\cv$-functor, it immediatly follows that $\ga$ is $\cv$-natural.

 \eqref{lemma(d):characterization-ff-vfunctor} $\Rightarrow$ \eqref{lemma(b):characterization-ff-vfunctor}. It follows from  the fact that $P_0 = \cv\textrm{-}\Cat (\ci , P )$ by definition.

\eqref{lemma(f):characterization-ff-vfunctor} $\Rightarrow$ \eqref{lemma(a):characterization-ff-vfunctor}.
Recall that there is a bijection
\begin{equation*}
    \ca\ni A \mapsto \functorA\in \cv\textrm{-}\Cat [\ci , \ca]
\end{equation*}
in which $\functorA :\ci \to \ca $ is the only $\cv$-functor from the unit $\cv$-category $\ci$ to $\ca$ such that
$\functorA 0 = A$. Moreover, for any $A, B\in\ca $, the hom-object $\ca(A,B)$  is the end  $\int_{\ci}\ca(\overline{A}-,\overline{B}-)$ which gives the hom-object $\cv-\Cat[\ci, \ca](\overline{A},  \overline{B})$.  We get that, for any $\cv$-functor $P:\ca\to \cb$, the morphism $P_{A,B}$  is essentially $\cv\textrm{-}\Cat [\ci , \ca ] (\functorA , \functorB ) $.

 Therefore $\cv\textrm{-}\Cat [\ci , P ]$ is $\cv$-fully faithful if and only if $P$ is $\cv$-fully faithful.

\eqref{lemma(a):characterization-ff-vfunctor} $\Rightarrow$ \eqref{lemma(f):characterization-ff-vfunctor}. Given a $\cv$-category $\cc$ and $\cv$-functors $F,G:\cc\to \ca$, we have that
$$\cv\textrm{-}\Cat[\cc, P] _{F,G} : \cv\textrm{-}\Cat [\cc,\ca](F,G)\to \cv\textrm{-}\Cat[\cb,\ca](PF,PG) $$
is, by definition, the morphism
\begin{equation}\label{end-of-v-natural-transformation-induced-by-P}
\int_{C\in \cc} P_{(FC,GC)} : \int_{C\in \cc}\ca (FC, GC)\to\int_{C\in \cc}\cb (PFC, PGC) \end{equation}
induced by the $\cvhf$natural transformation
between the $\cvhf$functors $\ca (F-, G-) $ and  $\cb (PF-, PG-)$
whose components are given by
\begin{equation}\label{v-natural-transformation-induced-by-P}
P_{FA,GB} : \ca (FA, GB)\to\cb (PFA, PGB) .
\end{equation}
Since $P$ is $\cv$-fully faithful, we have that \eqref{v-natural-transformation-induced-by-P}
is invertible and, hence, \eqref{end-of-v-natural-transformation-induced-by-P}
is invertible.
\end{proof}

From Lemma \ref{l:ff-enr}, we obtain:

\begin{lemma}\label{coro:iso-coherence-fullyfaithful-rightadjoin}
Given a $\cv$-adjunction $ \left( \varepsilon , \eta \right) : F\dashv G : \ca\to\cb $, the $\cv$-functor $G$ is $\cv$-fully faithful if and only if there is any (ordinary) natural isomorphism $$ F_0G_0 \rightarrow \id _{\ca _ 0} .$$
\end{lemma}
\begin{proof}
By Lemma \ref{l:ff-enr}, $G$ is $\cv$-fully faithful if and only if $G_0$ is fully faithful.  It is well known that $G_0$ is fully faithful in $\Cat $ if and only if the counit $\varepsilon _ 0 $ is invertible (see diagram \eqref{eq:adjunction-diagram-counit}), and, following~\cite[Lemma~1.3]{JM},  if and only if there is any
natural isomorphism between $F_0G_0$ and the identity\footnote{See \cite{JM} or \cite{FLN1} for further results on non-canonical isomorphisms.}.
\end{proof}

 On one hand, by Definition \ref{d:laxepi}, a $\cv$-functor $P:\ca\to \cb$ between small $\cv$-categories is said  a lax epimorphism in the 2-category $\cv$-$\Cat$ if the (ordinary) functor
$$\cv\text{-}\Cat(P, \cc): \cv\text{-}\Cat(\cb, \cc) \to \cv\text{-}\Cat(\ca, \cc)$$ is fully faithful, for all $\cv$-categories $\cc$.
On the other hand, the notion of $\cv$-fully faithful functor and Lemma \ref{l:ff-enr} inspire the following definition.

\begin{definition}\label{d:vlaxepi}
A $\cv$-functor $J:\ca\to \cb$ (between small $\cv$-categories) is a {\em $\cv$-lax epimorphism} if, for any $\cc $ in $\cv$-$\Cat$, the $\cv$-functor
$$\cv\text{-}\Cat \left[ J, \cc\right] : \cv\text{-}\Cat \left[ \cb, \cc\right] \to \cv\text{-}\Cat\left[ \ca, \cc \right]$$
is $\cv$-fully faithful.
\end{definition}

\begin{assumption}\label{assumption2}
Until now,  we are assuming that $\cv_0$, and then also the $\cv$-category $\cv$, is complete (Assumption \ref{assumption1}). From now on, we assume furthermore that $\cv_0$ is also cocomplete.
\end{assumption}

\begin{theorem}\label{t:V-lax}
 Given a $\cv$-functor $J:\ca \to \cb$ between small $\cv$-categories $\ca$ and $\cb$,   the following conditions are equivalent.
\begin{enumerate}[(a)]
\item $J$ is a $\cv$-lax epimorphism.\label{(g)t:V-lax}
\item $J$ is a lax epimorphism in the 2-category $\cv$-$\Cat$. \label{(h)t:V-lax}
\item The functor $\cv$-$\Cat(J, \cv): \cv$-$\Cat(\cb, \cv) \to \cv$-$\Cat(\ca, \cv)$ is fully faithful.\label{(i)t:V-lax}
\item The $\cv$-functor $\cv$-$\Cat[J, \cv]: \cv\textrm{-}\Cat[\cb, \cv] \to \cv$-$\Cat[\ca, \cv]$ is $\cv$-fully faithful. \label{(a)t:V-lax}
\item\label{item-t:Lan} There is a $\cv$-natural isomorphism $\Lan _J\cb (B, J - )\cong \cb (B, -) $ ($\cv$-natural in $B\in\cb ^\op $).\label{(e)t:V-lax}
\item The $\cv$-functor $\cv$-$\Cat[J, \cc]: \cv\textrm{-}\Cat[\cb, \cc] \to \cv$-$\Cat[\ca, \cc]$ is $\cv$-fully faithful for every (possibly large) $\cv$-category $\cc$.\label{(f)t:V-lax}
\end{enumerate}
\end{theorem}

\begin{proof}

\eqref{(g)t:V-lax} $\Rightarrow$ \eqref{(h)t:V-lax}.
It follows from the implication \eqref{lemma(a):characterization-ff-vfunctor} $\Rightarrow$ \eqref{lemma(b):characterization-ff-vfunctor} of Lemma \ref{l:ff-enr}.
Namely, given a (small) $\cv$-category $\cc$,
since $\cv\textrm{-}\Cat\left[  J , \cc \right] $ is $\cv$-fully faithful, we get that
$\cv\textrm{-}\Cat\left[  J , \cc \right] _ 0 = \cv\textrm{-}\Cat\left(  J , \cc \right) $ is fully faithful.

\eqref{(h)t:V-lax} $\Rightarrow$ \eqref{(i)t:V-lax}. Given any $\cv$-functors
$F,G : \cb\to \cv $, we denote by $ P : \cc\to \cv $ the full inclusion of the (small) sub-$\cv$-category
of $\cv $ whose objects are in the image of $F$ or in the image of $G$.

It should be noted that $\cv\textrm{-}\Cat\left(  J , \cc \right) _{F,G} $ is a bijection by hypothesis, and  $\cv\textrm{-}\Cat\left(  \ca ,  P\right) _{F,G}, \cv\textrm{-}\Cat\left(  \cb ,  P\right) _{F,G} $
are bijections since $P$ is $\cv$-fully faithful. Therefore, since
the diagram
\begin{equation}\label{eq:C-to-V}
\begin{tikzpicture}[x=5cm, y=1.5cm]
\node (a) at (0,0) {$\cv\textrm{-}\Cat\left(  \cb , \cc \right)\left( F, G \right) $};
\node (b) at (2,0) {$\cv\textrm{-}\Cat\left(  \ca , \cc \right)\left( F\cdot J , G\cdot J \right) $ };
\node (c) at (0,-1) {$\cv\textrm{-}\Cat\left(  \cb , \cv\right) \left( F, G \right) $ };
\node (d) at (2,-1) {$\cv\textrm{-}\Cat\left(  \ca  , \cv\right)\left( F\cdot J , G\cdot J \right) $ };
\draw[->] (c)--(d) node[midway,below] {$\cv\textrm{-}\Cat\left(  J , \cv\right) _{F,G} $ };
\draw[->] (a)--(c) node[midway,left] {$\cv\textrm{-}\Cat\left(  \cb ,  P\right) _{F,G} $ };
\draw[->] (b)--(d) node[midway,right] {$\cv\textrm{-}\Cat\left(  \ca ,  P\right) _{F,G} $ };
\draw[->] (a)--(b) node[midway,above] {$\cv\textrm{-}\Cat\left(  J , \cc \right) _{F,G} $};
\end{tikzpicture}
\end{equation}
commutes,  we conclude that $\cv\textrm{-}\Cat\left( J, \cv\right) _{F,G}$ is also a bijection.
This proves that $\cv\textrm{-}\Cat\left( J, \cv\right) $ is fully faithful.

\eqref{(i)t:V-lax} $\Rightarrow$ \eqref{(a)t:V-lax}.
Since $\cv $ is complete, we have that $\cv\textrm{-}\Cat\left[ J, \cv\right] $ has a right $\cv$-adjoint given by the (pointwise) Kan extensions $\Ran _J $.
Therefore, assuming that $\cv\textrm{-}\Cat\left( J, \cv\right) $ is  fully faithful, we conclude that $\cv\textrm{-}\Cat\left[ J, \cv\right] $
is $\cv$-fully faithful by Lemma \ref{l:ff-enr}.

\eqref{(a)t:V-lax} $\Rightarrow$ \eqref{(e)t:V-lax}.
Since $\cv $ is cocomplete, we have that $\Lan _J\dashv  \cv\textrm{-}\Cat \left[ J, \cv \right] $. Therefore, assuming that $\cv\textrm{-}\Cat \left[ J, \cv \right] $ is $\cv$-fully faithful, we have the $\cv$-natural isomorphism  $\epsilon : \Lan _ J \left( - \cdot J\right)\cong \id _{\cv\textrm{-}\Cat [\cb , \cv ] } $ given by the counit.

Denoting by $\YonedaE _{\cb ^\op } $ the Enriched Yoneda Embedding (see, for instance, \cite[2.4]{kelly82}), we have that $\epsilon^{-1}\ast \id _{\YonedaE  _{\cb ^\op } }$  gives
an isomorphism $\Lan _J\cb (B, J - )\cong \cb (B, -) $ ($\cv$-natural in $B\in\cb ^\op $).

\eqref{(e)t:V-lax} $\Rightarrow$ \eqref{(f)t:V-lax}.
Let $\cc $ be any (possibly large) $\cv$-category. We consider the $\cv $-functor $\cv\textrm{-}\Cat[ J , \cc] $
and its factorization
\begin{equation}
\begin{tikzpicture}[x=7cm, y=1.6cm]
\node (a) at (0,0) {$\cv\textrm{-}\Cat\left[  \cb , \cc \right] $};
\node (b) at (1,0) {$\imagefactorizationofprecomposition $ };
\node (c) at (1,-1) {$\cv\textrm{-}\Cat\left[  \ca , \cc\right]  $ };
\draw[->] (a)--(c) node[midway,below left] {$\cv\textrm{-}\Cat\left[  J , \cc\right] $ };
\draw[->] (a)--(b) node[midway,above] {$\imagefactorizationofprecompositionbo $};
\draw[->] (b)--(c);
\end{tikzpicture}
\end{equation}
into a bijective on objects $\cv $-functor $\imagefactorizationofprecompositionbo $
and the $\cv$-full inclusion
$\imagefactorizationofprecomposition\to \cv\textrm{-}\Cat[ \ca , \cc] $  of the sub-$\cv$-category $\imagefactorizationofprecomposition $  whose objects are in the image of $\cv\textrm{-}\Cat\left[  J , \cc\right] $.
We prove below that  $ \cv\textrm{-}\Cat[ J , \cc]  $ is $\cv $-fully faithful by proving that $\imagefactorizationofprecompositionbo  $ is $\cv $-fully faithful.

Given any $\cv $-functor $G: \ca\to\cc $ in $\imagefactorizationofprecomposition $, we have that $G = FJ $ for some
$F: \cb\to \cc $.
Since  $\Lan _ J \cb (B , J-)\cong \cb (B , -) $, we conclude that $\lim \left( \Lan _ J \cb (B , J-), F   \right) $ exists and, moreover, we have the isomorphisms
\begin{equation}\label{eq:phase1-isomorphic-LanJB(b,J-) cong B(b,-)}
\lim \left( \Lan _ J \cb (B , J-), F   \right)\cong \lim \left(  \cb (B , -), F   \right) \cong F(B)
\end{equation}
by the (strong) Enriched Yoneda Lemma (see \cite[Sections~2.4 and 4.1]{kelly82}).

Since $\lim \left( \Lan _ J \cb (B , J-), F   \right) $ exists, it follows as a consequence of
the universal property of left Kan extensions that
$\lim\left( \cb (B , J-), F\cdot J  \right)$ exists and is isomorphic to $\lim \left( \Lan _ J \cb (B , J-), F   \right) $
(see \cite[Proposition~4.57]{kelly82}). Therefore, by \eqref{eq:phase1-isomorphic-LanJB(b,J-) cong B(b,-)} and by the formula for pointwise right Kan extensions (see  \cite[Theorem~I.4.2]{Dubuc-KanExtensions} or, for instance, \cite[Theorem~4.6]{kelly82}), we conclude that $\Ran _ J \left( F\cdot J \right)$ exists and
we have the isomorphism
\begin{equation}\label{eq:right-Kan-extension-of-functors-in-the-image}
\Ran _ J \left( F\cdot J \right) B\cong \lim\left( \cb (B , J-), F\cdot J  \right)\cong \lim \left( \Lan _ J \cb (B , J-), F   \right)\cong \lim \left(  \cb (B , -), F   \right) \cong F(B)
\end{equation}
$\cv$-natural in $B\in\cb$ and $F\in \cv\textrm{-}\Cat[\cb, \cc]$.

Since we proved that $ \Ran _ J \left( F\cdot J \right) $ exists for any $G = F\circ J$ in $\imagefactorizationofprecomposition $, we conclude that $\imagefactorizationofprecompositionbo $ has a right $\cv $-adjoint, which we may denote by $\Ran _J $ by abuse of language. Finally,  by the natural isomorphism \eqref{eq:right-Kan-extension-of-functors-in-the-image} and Lemma \ref{coro:iso-coherence-fullyfaithful-rightadjoin}, we conclude that $\imagefactorizationofprecompositionbo $
is $\cv$-fully faithful.

\eqref{(f)t:V-lax} $\Rightarrow$ \eqref{(g)t:V-lax}.
Trivial.
\end{proof}

\begin{remark}
For $\cv=\Set$,
the equivalence \eqref{(h)t:V-lax} $\Leftrightarrow$ \eqref{(i)t:V-lax} of Theorem \ref{t:V-lax} was given in \cite[Theorem~1.1]{AeBSV}.
\end{remark}

\begin{example}Let $V$ be a frame, that is, a complete lattice satisfying the infinite distributive law of $\wedge$ over $\vee$, regarded as a symmetric monoidal closed category, where the multiplication and the unit  object are $\wedge$ and $1$, respectively. Thus, $V$ is a quantale, see \cite{Stubbe} and \cite{HST}. The hom-objects of the $V$-category $V$ are given by $V(a,b)=a\to b$, where $\to$ is the Heyting operation. For every $V$-category $X$, the hom-objects of $[X,V]$ are given, for every pair of $V$-functors $f,g:X\to V$,  by $$[X,V](f,g)=\displaystyle{\bigwedge_{x\in X}\big(f(x)\to g(x)\big)}.$$
Following Definition \ref{d:vlaxepi} and the equivalence  (a)$\Leftrightarrow$(b) of Theorem \ref{t:V-lax}, we see that a $V$-functor $j:X\to Y$ is a lax epimorphism in the 2-category $V$-$\Cat$ if and only if, for every pair of $V$-functors $f,g:Y\to V$, $$\displaystyle{\bigwedge_{y\in Y}\big(f(y)\to g(y)\big)=\bigwedge_{x\in X}\big(fj(x)\to gj(x)\big)}.$$
\end{example}

\begin{example}
Two important examples covered by Theorem \ref{t:V-lax} are discrete fibrations and split fibrations. Every functor $J:\ca \to \cb$ between (small) ordinary categories induces a functor $J^\ast :  \DisFib\left( \cb\right) \to \DisFib\left( \ca\right) $   between the categories of discrete fibrations, and a functor $\underline{J}^\ast : \Fib\left(\cb\right) \to \Fib\left(\ca\right) $ between the $2$-categories of split fibrations.   By the Grothendieck construction, these two functors are essentially the precomposition functors  $\Cat\left(J, \Set \right): \Cat\left( \cb , \Set \right)\to\Cat\left( \ca , \Set \right)$ and $\Cat$-$\Cat\left[J, \Cat \right] : \Cat\textrm{-}\Cat\left[ \cb , \Cat \right]\to\Cat\textrm{-}\Cat\left[ \ca , \Cat \right]$, respectively.  It is easy to see that $J$ is a lax epimorphism in the 2-category $Cat$ if and only if it is a lax epimorphism in the 2-category $\Cat$-$\Cat$, when regarded as a 2-functor between locally discrete categories.  This actually follows from the fact that we have a $2$-adjunction satisfying both conditions of Lemma \ref{lem:2-adjunction-preservation-of-lax-epi}, where the left $2$-adjoint is given by the inclusion $\Cat\to\Cat\textrm{-}\Cat $.  Thus, using the equivalence \eqref{(h)t:V-lax} $\Leftrightarrow$ \eqref{(a)t:V-lax} of Theorem \ref{t:V-lax}, we conclude that $J^{\ast}$ is fully faithful if and only if $\underline{J^{\ast}}$  is $\cv$-fully faithful, if and only if $J$ is a lax epimorphism in $\Cat$.
\end{example}

\begin{lemma}[Duality]\label{rem:duality-lax-epimorphisms-Vcat}
A morphism $J:\ca \to \cb$ is a lax epimorphism in $\cv\textrm{-}\Cat$
if and only if $J ^\op :\ca ^\op \to \cb ^\op$ is a lax epimorphism in $\cv\textrm{-}\Cat$ as well.
\end{lemma}
\begin{proof}
Indeed, since the $2$-functor $\op :\cv\textrm{-}\Cat\to \cv\textrm{-}\Cat ^\co $ is invertible, it takes lax epimorphisms to lax epimorphisms.
Thus,  $J$ is a lax epimorphism in  $\cv\textrm{-}\Cat$ if, and only if, $\op  \left( J\right) $ is a lax epimorphism in $\cv\textrm{-}\Cat ^\co $ which, by Remark \ref{r:ff}, holds if and only if $J^\op $ is a lax epimorphism in  $\cv\textrm{-}\Cat$.

Therefore, assuming that $\cv _ 0$ is complete and cocomplete,
\begin{center}
$J$ is a $\cv$-lax epimorphism \, $\Leftrightarrow$ \, $J^\op $  is a $\cv$-lax epimorphism
\end{center}
by Theorem \ref{t:V-lax}.
\end{proof}

Recall that a $\cv$-functor $J: \ca\to\cb $ between small $\cv$-categories is \textit{$\cv$-dense} if and only if
its density comonad $\Lan _ J J $ is isomorphic to the identity on $\cb$ (see \cite[Theorem~5.1]{kelly82}).
Dually, $J $ is  \textit{$\cv$-codense} if and only if
the right Kan extension
$ \Ran _ J J $ is the identity. (Several concrete examples of ($\cv$-)codensity monads are given in \cite{AS21}.)

We say that $J$ is \textit{absolutely $\cv$-dense} if it is $\cv$-dense and $\Lan _ J J $
is preserved by any $\cv$-functor $F: \cb\to\cv $. Dually, we define absolutely $\cv$-codense $\cv$-functor.

The following characterization of lax epimorphisms as absolutely dense functors was given in \cite{AeBSV} for $\cv=\Set$:

\begin{theorem}\label{laxepi-if-and-only-if-absolutelydense}
Given a $\cv$-functor $J:\ca \to \cb$ between small $\cv$-categories $\ca$ and $\cb$,   the following conditions are equivalent.
\begin{enumerate}[(a)]
\item $J$ is a $\cv$-lax epimorphism.\label{(a):laxepi-if-and-only-if-absolutelydense}
\item $J$ is absolutely $\cv$-dense.\label{(b):laxepi-if-and-only-if-absolutelydense}
\item $J$ is absolutely $\cv$-codense.\label{(c):laxepi-if-and-only-if-absolutelydense}
\end{enumerate}
\end{theorem}

\begin{proof}
\eqref{(a):laxepi-if-and-only-if-absolutelydense} $\Rightarrow $ \eqref{(b):laxepi-if-and-only-if-absolutelydense}. Assume that  $J$ is a $\cv$-lax epimorphism.
By \eqref{(e)t:V-lax} of Theorem \ref{t:V-lax},  we have that $\cb (B, -)\cong \Lan _ J \cb (B, J-)$. Hence, since
$\lim \left(  \cb (B, -) , \id _\cb  \right)\cong B $ exists by the (strong) Enriched Yoneda Lemma,
we have that $ \lim \left( \Lan _ J \cb (B, J-) , \id _\cb  \right) $ exists and is isomorphic to $\lim \left(  \cb (B, -) , \id _\cb  \right)\cong B $ (in which isomorphisms are always $\cv$-natural in $B$).

Moreover, from the existence of $ \lim \left( \Lan _ J \cb (B, J-) , \id _\cb  \right) $, we get that
$ \lim \left( \cb (B, J-) , J  \right)$ exists and is isomorphic to $\lim \left( \Lan _ J \cb (B, J-) , \id _\cb  \right)\cong B $ (see \cite[Proposition~4.57]{kelly82}).

Finally, then, from the formula for pointwise right Kan extensions and the above, we get the $\cv$-natural isomorphisms (in $B\in \cb $)
$$ B \cong \lim \left(  \cb (B, -) , \id _\cb  \right)\cong \lim \left( \Lan _ J \cb (B, J-) , \id _\cb  \right) \cong \lim \left( \cb (B, J-) , J  \right) \cong  \Ran _J J(B) . $$
This proves that $\Ran _J J $ is the identity on $\cb$. That is to say, $J$ is $\cv$-codense.

Moreover, assuming that $J $ is a $\cv$-lax epimorphism, by Lemma \ref{rem:duality-lax-epimorphisms-Vcat}, $J^\op$ is a $\cv$-lax epimorphism and, hence, by the proved above, $J^\op$ is $\cv$-codense. Therefore $J$ is $\cv$-dense.

By  \eqref{(a)t:V-lax} of Theorem \ref{t:V-lax},  we have that $\cv\textrm{-}\Cat\left[ J , \cv \right] $ is $\cv $-fully faithful. Since $\cv $ is cocomplete, we get that  $\Lan _ J $ exists and there is an isomorphism  $\Lan _ J \left( F\cdot J\right)\cong F $, $\cv$-natural in $F\in \cv\textrm{-}\Cat[\cb, \cv] $, given by the counit of $\Lan _ J \dashv \cv\textrm{-}\Cat\left[ J , \cv \right] $. This shows that $\Lan _J J $ is preserved by any $\cv$-functor $F: \cb\to\cv $.

\eqref{(b):laxepi-if-and-only-if-absolutelydense} $\Rightarrow $ \eqref{(a):laxepi-if-and-only-if-absolutelydense}.
Assume that $J$ is absolutely $\cv $-dense.
We conclude that there is a natural isomorphism $ \Lan _ J \left( F\cdot J\right)\cong F $. Therefore, by Lemma \ref{l:ff-enr}, we conclude that
 $\cv\textrm{-}\Cat\left[ J , \cv \right] $
is $\cv $-fully faithful. By Theorem \ref{t:V-lax}, this proves that $J$
is a $\cv$-lax epimorphism.

\eqref{(a):laxepi-if-and-only-if-absolutelydense} $\Leftrightarrow $ \eqref{(c):laxepi-if-and-only-if-absolutelydense}.
By Lemma \ref{rem:duality-lax-epimorphisms-Vcat} and by the proved above, we conclude that
\begin{center}
$J$ is a $\cv$-lax epimorphism \, $\Leftrightarrow $ \,   $J^\op $ is absolutely $\cv$-dense \, $\Leftrightarrow $  $J$ is absolutely $\cv$-codense.
\end{center}
\end{proof}

\begin{remark}[Counterexample: Density and Codensity]
Of course, density and codensity are not enough for a functor to be a lax epimorphism: for $\mathsf{1} $ the  terminal object  in $\Cat $, the functor  $J :\mathsf{1}\sqcup \mathsf{1}\to\mathsf{1}  $ is dense and codense, but not a lax epimorphism. Moreover,
$\Ran _{J} J$  (respectively, $\Lan _{J} J $)  is preserved by $F : \mathsf{1}\to \Set  $  if and only if the image of $F$ is a preterminal object, \textit{i.e.} the terminal set $\mathsf{1} $ (respectively, a preinitial object, \textit{i.e.} the empty set $\emptyset $ ); see \cite[Remark~4.14]{FLN2} and \cite[Remark~4.5]{FLN1}.
\end{remark}


Let us recall from \cite{SW} that a 1-cell $J:A\to B$ of $\Cat$ is an initial functor precisely when
\begin{equation*}
\begin{tikzpicture}[x=1.9cm, y=1.9cm]
\node (a) at (0,0) {$\ca $};
\node (b) at (1, 0) { $\cb$ };
\node (c) at (1,-1) {$\Set$};
\node (d) at (0,-1) {$ \cb $};
\draw[->] (a)--(b) node[midway,above] {$ J  $};
\draw[->] (b)--(c) node[midway,right] {$ \ast\!\downarrow  $};
\draw[->] (a)--(d) node[midway,left] {$ J  $};
\draw[->] (d)--(c) node[midway,below] {$ \ast\!\downarrow  $};
\draw[double,->] (0.3,-0.55)--(0.7,-0.55) node[midway,above] {$\id $};
\end{tikzpicture}
\end{equation*}
exhibits $\ast\!\!\downarrow$ as a left Kan extension of $\ast \!\!\downarrow \!\cdot E$ along $E$, where $\ast\!\!\downarrow$ is the constant functor on the terminal object $1$.
This characterization is a key point for the description of the comprehensive factorization system by means of left Kan extensions given by Street and Walters. The equivalence (\ref{(h)t:V-lax}) $\Leftrightarrow$ (\ref{item-t:Lan}) of Theorem \ref{t:V-lax} shows that  lax epimorphisms in $\cv$-$\Cat$ have an analogous presentation. More precisely:

A $\cv$-functor $\ca\xrightarrow{J}\cb$ is a lax epimorphism in $\cv$-$\Cat$ precisely when
\begin{equation}\label{eq:extra}
\begin{tikzpicture}[x=2cm, y=2cm]
\node (a) at (0,1) {$\ca $};
\node (b) at (1,1) {$\cb $ };
\node (d) at (0,0) {$\cb $ };
\node (e) at (1,0) {$ [\cb^{\op},\cv]  $};
\draw[->] (a)--(b) node[midway,above] {$ J $ };
\draw[->] (a)--(d) node[midway,left] {$J   $ };
\draw[->] (d)--(e) node[midway,below] {$Y   $ };
\draw[->] (b)--(e) node[midway,right] {$Y   $ };
\draw[double,->] (0.3,0.45)--(0.7,0.45) node[midway,above] {$\id $};
\end{tikzpicture}
\end{equation}
exhibits $Y$ as a left Kan extension of $Y  J$ along $J$.

In Section 4, we gave a concrete description of the orthogonal $LaxEpi$-factorization system in $\Cat$ (Theorem \ref{t:f-Cat}), in particular  $\Cat$ satisfies the hypotheses of Theorem \ref{t:ofs}.   The  characterization of the lax epimorphisms in $\cv$-$\Cat$ given by (\ref{eq:extra})   suggests the possibility of describing the orthogonal $LaxEpi$-factorization system in $\cv$-$\Cat$ by means of left Kan extensions in the style of \cite{SW} for the comprehensive factorization. This is far from having an obvious path, and will be the subject of future work.

\end{document}